\pdfoutput=1
\RequirePackage{ifpdf}
\ifpdf 
\documentclass[pdftex]{sigma}
\else
\documentclass{sigma}
\fi

\usepackage[all]{xy}

\numberwithin{equation}{section}

\newtheorem{Theorem}{Theorem}[section]
\newtheorem{Lemma}[Theorem]{Lemma}
\newtheorem{Proposition}[Theorem]{Proposition}
\newtheorem{Conjecture}[Theorem]{Conjecture}
\newtheorem{Question}[Theorem]{Question}
\newtheorem{Claim}[Theorem]{Claim}
 { \theoremstyle{definition}
\newtheorem{Example}[Theorem]{Example}
\newtheorem{Remark}[Theorem]{Remark} }

 \DeclareMathOperator{\diam}{diam}

 \newcommand{\Rm}{\operatorname{\bf Rm}}
 \newcommand{\Rc}{\operatorname{\bf Rc}}
 \newcommand{\Sc}{\mathbf{R}}

\newcommand{\rank}{\operatorname{rank}}

\begin{document}
\allowdisplaybreaks

\newcommand{\arXivNumber}{2008.12419}

\renewcommand{\thefootnote}{}

\renewcommand{\PaperNumber}{123}

\FirstPageHeading

\ShortArticleName{Collapsing Geometry with Ricci Curvature Bounded Below and Ricci Flow Smoothing}

\ArticleName{Collapsing Geometry with Ricci Curvature\\ Bounded Below and Ricci Flow Smoothing\footnote{This paper is a~contribution to the Special Issue on Scalar and Ricci Curvature in honor of Misha Gromov on his 75th Birthday. The full collection is available at \href{https://www.emis.de/journals/SIGMA/Gromov.html}{https://www.emis.de/journals/SIGMA/Gromov.html}}}

\Author{Shaosai HUANG~$^\dag$, Xiaochun RONG~$^\ddag$ and Bing WANG~$^\S$}

\AuthorNameForHeading{S.~Huang, X.~Rong and B.~Wang}

\Address{$^\dag$~Department of Mathematics, University of Wisconsin-Madison, Madison, WI 53706, USA}
\EmailD{\href{mailto:sshuang@math.wisc.edu}{sshuang@math.wisc.edu}}
\URLaddressD{\url{https://sites.google.com/a/wisc.edu/shaosai-huang}}

\Address{$^\ddag$~Department of Mathematics, Rutgers University, New Brunswick, NJ 08854, USA}
\EmailD{\href{mailto:rong@math.rutgers.edu}{rong@math.rutgers.edu}}

\Address{$^\S$~Institute of Geometry and Physics, and School of Mathematical Sciences,\\
\hphantom{$^\S$}~University of Science and Technology of China, Hefei, Anhui Province, 230026, China}
\EmailD{\href{mailto:topspin@ustc.edu.cn}{topspin@ustc.edu.cn}}
\URLaddressD{\url{http://staff.ustc.edu.cn/~topspin/index.html}}

\ArticleDates{Received August 30, 2020, in final form November 23, 2020; Published online November 30, 2020}

\Abstract{We survey some recent developments in the study of collapsing Riemannian manifolds with Ricci curvature bounded below, especially the locally bounded Ricci covering geometry and the Ricci flow smoothing techniques. We then prove that if a Calabi--Yau manifold is sufficiently volume collapsed with bounded diameter and sectional curvature, then it admits a Ricci-flat K\"ahler metric
together with a compatible pure nilpotent Killing structure: this is related to an open question of Cheeger, Fukaya and Gromov.}


\Classification{53C21; 53C23; 53E20}

\renewcommand{\thefootnote}{\arabic{footnote}}
\setcounter{footnote}{0}

\section{Introduction}
In the seminal work \cite{Gromov78b}, Gromov discovered a gap phenomenon for the
sectional curvature (denoted by $\mathbf{K}_g$ for a smooth Riemannian metric~$g$) to detect the infranil manifold structure (see also \cite{BK81, Ruh82}):

\begin{Theorem}[Gromov's almost flat manifold theorem, 1978]\label{thm:
Gromov78}
There is a dimensional constant $\varepsilon(m)\in (0,1)$
such that if a closed $m$-dimensional Riemannian manifold $(M,g)$ satisfies
\begin{gather}\label{eqn: AF_Rm}
\diam(M,g)^2\max_{\wedge^2TM} |\mathbf{K}_{g} | \le \varepsilon^2 ,
\end{gather}
then $M$ is diffeomorphic to an infranil manifold.
\end{Theorem}
Here we say that $M$ is an \emph{infranil manifold} if on the universal covering
$\tilde{M}$ of~$M$ there is a flat connection with parallel torsion, defining a~simply connected nilpotent Lie group structure~$N$ on~$\tilde{M}$ such that
$\pi_1(M)$ is a sub-group of $N\rtimes \operatorname{Aut}(N)$ with
$ [\pi_1(M)\colon \pi_1(M)\cap N ]<\infty$ and $\rank \pi_1(M)=m$~-- in the case of Gromov's
almost flat manifold theorem, it is also shown that such index has a uniform
dimensional upper bound $C(m)$.

Ever since its birth, Gromov's almost flat manifold theorem has inspired the
research of Riemannian geometers by two themes of generalizations. One theme
is to find parametrized versions of Theorem~\ref{thm: Gromov78}, as indicated by
Fukaya's fiber bundle theorem~\cite{Fukaya87ld}: if a Riemannian manifold with
bounded diameter and sectional curvature is sufficiently Gromov--Hausdorff close
to a lower dimensional one with bounded geometry, then it is diffeomorphic to
the total space of a smooth family of almost flat manifolds parametrized over
the lower dimensional one (see Theorem~\ref{thm: Fukaya1}). The generalized
version of this theorem (see Theorem~\ref{thm: singular_fibration}), combined
with the more intrinsic approach of Cheeger and Gromov~\cite{CGI, CGII} on
$F$-structures of positive rank, whose existence is equivalent to the existence
of a one-parameter family of Riemannian metrics collapsing with bounded
sectional curvature, nurtured the rich and splendent theory of the collapsing
geometry with bounded sectional curvature, notably the construction (in~\cite{CFG92}) and various applications of the \emph{nilpotent Killing
structure} (see Theorem~\ref{thm: CFG}). This theory,
as we will briefly recall in Section~\ref{section2.1}, is mainly developed through the works of
Cheeger and Gromov \cite{CGI, CGII}, Fukaya \cite{Fukaya87mGH, Fukaya87ld,
Fukaya88, Fukaya89}, and Cheeger, Fukaya and Gromov \cite{CFG92}.

The other theme of generalization focuses on weakening the curvature assumption~(\ref{eqn: AF_Rm}) to lower (Bakry--\'Emery) Ricci curvature bounds, as
examplified by the Colding--Gromov gap theorem: if a \emph{Ricci} almost
non-negatively curved manifold of unit diameter has its first Betti number
equal to its dimension, then the manifold is diffeomorphic to a flat torus
(see Theorem~\ref{thm: CG97}). Obviously, the weaker curvature assumption alone
is insufficient to conclude the infranil manifold structure, and certain extra
assumptions are necessary~-- just as the case of the Colding--Gromov gap
theorem; see also \cite{DWY96, HKRX18, Kapovitch19, PWY99} and Section~\ref{section2.2} for a~brief overview.

The two streams of research become confluent in the study of collapsing
Riemannian manifolds with Ricci curvature bounded below. The exploration in
this field is driven not just by its intrinsic merit of finding weaker
curvature assumptions, but also externally by the study of K\"ahler
geometry and mathematical physics: a major open problem in string theory is the
SYZ conjecture~\cite{SYZ}, which could be formulated as understanding the
collapsing geometry of Ricci-flat K\"ahler metrics: see, e.g., \cite{KS01, Li19,
Tossatti20}; see also \cite{ GTZ13, GTZ16,GW00, HSVZ18} for some examples
concerning such phenomena.

The weaker curvature assumption causes tremendous
difficulties for our understanding, and it is natural to start with some
extra assumptions. For instance, one could impose some extra conditions on
topology such as the first Betti numbers: in \cite{HW20a}, it is shown
by the first- and third-named authors that if a Riemannian manifold with Ricci
curvature bounded below is sufficiently Gromov--Hausdorff close to a lower
dimensional one with bounded geometry, and the difference of their first Betti
numbers is equal to their dimensional difference, then the higher dimensional
manifold is diffeomorphic to a torus bundle over the lower dimensional one
(see Theorem~\ref{thm: HW20a}). This theorem generalizes Fukaya's fiber bundle
theorem and the Colding--Gromov gap theorem simultaneously, in setting of
collapsing Riemannian manifolds with Ricci curvature bounded below -- related
results will be surveyed in Section~\ref{section2.3}.

Since the collapsing phenomenon with bounded sectional curvature is essentially
due to the abundant local symmetry encoded in the local fundamental group, a
very natural condition is to assume that the local universal covering space
around every point is non-collapsing: this is the program initiated by the
second-named author and his collaborators in~\cite{HKRX18} to systematically
investigate those manifolds with locally bounded Ricci covering geometry, which
we will discuss in Section~\ref{section3}.
In fact, the major effort in proving Theorem~\ref{thm: HW20a} is devoted to
decoding the topological information associated with the first Betti numbers
and show that the higher dimensional manifold has locally bounded Ricci
covering geometry.

A key issue in the study of collapsing Riemannian manifolds with Ricci curvature
bounded below is the low regularity of the metric due to the (weaker) Ricci
curvature assumption. This confines our understanding on the finer structures
of the collapsing geometry, and suitable smoothing of the given metric is
usually inevitable -- in Section~\ref{section4}, we will survey the relevant Ricci flow
smoothing techniques for locally collapsing manifolds with Ricci curvature
bounded below.

In fact, the Ricci flow smoothing technique also enhances our understanding on
the classical theory of collapsing with bounded sectional curvature. Cheeger,
Fukaya and Gromov asked in~\cite{CFG92} the following question which remains
open today.
\begin{Question}\label{qst: question}
It is known that given a complete Riemannian manifold $(M,g)$ with
sectional curvature uniformly bounded between $\pm 1$, for any small
$\varepsilon>0$ there is a regular $(\rho, k)$-round metric $g_{\varepsilon}$
and nilpotent Killing structure $\mathfrak{N}$ compatible with
$g_{\varepsilon}$, such that $\left\|g-g_{\varepsilon}\right\|_{C^1}<
\varepsilon$. Now if the initial metric $g$ is assumed to be K\"ahler or
Einstein, can we find $g_{\varepsilon}$ in the same category?
\end{Question}
In the last section of this note, we will prove, based on Ricci flow techniques,
that if a Ricci-flat K\"ahler metric is very collapsed with bounded diameter and
sectional curvature, then the approximating metric compatible with a nilpotent
Killing structure may indeed be found as a~nearby Ricci-flat K\"ahler metric.
Notice that the approximating metric~$g_{\varepsilon}$ obtained from
\cite[Theorem~1.7]{CFG92} is not necessarily K\"ahler or Ricci-flat, but we
manage to evolve it along the Ricci flow to find a desirable one. Since the
Ricci flow respects local isometries, the evolved metrics are compatible with the
original nilpotent Killing structure~$\mathfrak{N}$. We wish our result could
cast some light on the general case of Cheeger, Fukaya and Gromov's open
question.

\section{Collapsing geometry with sectional or Ricci curvature bounds}\label{section2}
In this section we give a short survey of the two directions generalizing
 Theorem~\ref{thm: Gromov78} -- the collapsing geometry with bounded
sectional curvature and the almost flatness characterized by weaker curvature
conditions -- as well as the study of collapsing geometry with only Ricci
curvature bounded below. While there have been comprehensive surveys on
the theory of collapsing geometry \cite{Fukaya06, Rong07}, we still
briefly go through some classical theorems so as to put the study of collapsing
geometry with Ricci curvature bounded below in the historical context.

\subsection{Collapsing with bounded sectional curvature}\label{section2.1}
Gromov's almost flat manifold theorem, when embedded in the framework of the
coarse geometry on the space of all Riemannian manifolds (see \cite{GLP}),
opened a new chapter in the study of Riemannian geometry: the collapsing
geometry of Riemannian manifolds with bounded curvature. In this note, we
consider the following collections of Riemannian manifolds:
\begin{enumerate}\itemsep=0pt
\item[1)] $\mathcal{M}_{\rm Rm}(m,D)$ denoting the collection of $m$-dimensional
Riemannian manifolds with sectional curvature bounded between $\pm 1$,
and diameter bounded from above by $D\ge 1$;
\item[2)] $\mathcal{M}_{\rm Rm}(m,D,v)$ denoting the sub-collection of
$\mathcal{M}_{\rm Rm}(m,D)$ with volume bounded below by $v>0$.
 \end{enumerate}
 Equipped with the Gromov--Hausdorff
 topology, the moduli space $\mathcal{M}_{\rm Rm}(m,D)$ is pre-com\-pact~\cite{GLP}.
 Based on the work of Cheeger (see~\cite{Cheeger70, GreenWu}), the
 sub-collection $\mathcal{M}_{\rm Rm}(m,D,v)$ is not just compact in the
 Gromov--Hausdorff topology, but also has only finitely many diffeomorphism
 classes. On the contrary, if we consider a sequence $\{(M_i,g_i)\}$
 in $\mathcal{M}_{\rm Rm}(m,D)$, then under the Gromov--Hausdorff topology, it is
 possible that $M_i\xrightarrow{\rm GH} N$, for some lower dimensional manifold
 $(N,h)$ in $\mathcal{M}_{\rm Rm}(k,D,v)$ with $k<m$. In this case, we say that
 $\left\{(M_i,g_i)\right\}$ \emph{collapses to}~$N$ \emph{with bounded curvature
 and diameter}. In \cite{Fukaya87ld} it is shown that such situation
 can only occur when $M_i$ are infranil fiber bundles over~$N$.
 \begin{Theorem}[Fukaya's fiber bundle theorem, 1987]\label{thm: Fukaya1}
 Given $D\ge 1$ and $v>0$, there is a~uniform constant $\varepsilon(m,v)\in
 (0,1)$ such that if $(M,g)\in \mathcal{M}_{\rm Rm}(m,D)$ and $(N,h)\in
 \mathcal{M}_{\rm Rm}(k,D,v)$ with $k\le m$ satisfy $d_{\rm GH}(M,N)<\delta$ for some
 $\delta<\varepsilon(m,v)$, then there is a $C^1$ submersion $f\colon M\to N$ such
 that
 \begin{enumerate}\itemsep=0pt
 \item[$1)$] $f$ is an almost Riemannian submersion, i.e., ${\rm e}^{-\Psi_F(\delta|m,v)}|\xi|_g \le
|f_{\ast}\xi|_h\le {\rm e}^{\Psi_F(\delta|m,v)}|\xi|_g$ for any $\xi\perp
\ker f_{\ast}$, with $\Psi_F(\delta|m,v)\in (0,1)$ satisfying $\lim\limits_{\delta
\to \infty} \Psi_F(\delta|m,v)=0$; and
\item[$2)$] the fiber of $f$ is diffeomorphic to an infranil manifold.
 \end{enumerate}
 \end{Theorem}
This theorem describes the diffeomorphism type of those
sufficiently collapsed manifolds in $\mathcal{M}_{\rm Rm}(m,D)$ by those ``minimal
models'' in $\mathcal{M}_{\rm Rm}(k,D,v)$, as long as we can find such a lower
dimensional model space. In general however, we cannot expect a sequence in
$\mathcal{M}_{\rm Rm}(m,D)$ to collapse to an element
in $\mathcal{M}_{\rm Rm}(k,D,v)$. We will refer to those Hausdorff $k$-dimensional
($k<m$) metric spaces arising as the Gromov--Hausdorff limits of sequences in
$\mathcal{M}_{\rm Rm}(m,D)$ as the \emph{collapsing limit spaces}. The local
structure of such spaces is described in \cite[Theorem~0.5]{Fukaya88}: for any
point $x$ in a~collapsing limit space $(X,d_X)$, there is an open neighborhood
$V$ of $x$, a~Lie group $G_x$ admitting a faithful representation to ${\rm O}(n)$ (for
some $n\le m$), and a $G_x$-invariant Riemannian metric $\tilde{g}$ on an open
neighborhood $U$ of $\vec{o}\in \mathbb{R}^n$, such that the identity component
of $G_x$ is isomorphic to a torus and $(V,d_X)\equiv (U,\tilde{g})\slash G_x$.
It is consequently shown in \cite{CFG92, Fukaya88, Fukaya89} that a manifold $M\in \mathcal{M}_{\rm Rm}(m,D)$ sufficiently
Gromov--Hausdorff close to a collapsing limit space $X$ exhibits a singular
fibration over $X$:
\begin{Theorem}[singular fibration, 1988--1992] \label{thm: singular_fibration}
Given $D\ge 1$, there are uniform constants $\varepsilon(m,D)$ and $c(m)>0$ to
the following effect: if $(M,g)\in \mathcal{M}_{\rm Rm}(m,D)$ satisfies
$|M|_g<\varepsilon\le \varepsilon(m,D)$, then the frame bundle $F(M)$ of $M$,
equipped with the canonical metric $\bar{g}$, is Gromov--Hausdorff close to some
$(Y,h)\in \mathcal{M}_{\rm Rm}(m',D',v')$ with $m'<m+\frac{1}{2}m(m-1)$ and
$D',v'>0$ determined by $m$ and $D$, such that ${\rm O}(m)$ acts isometrically on
$(Y,h)$ and there is an ${\rm O}(m)$-equivariant fiber bundle $\bar{f}\colon
(F(M),{\rm O}(m))\to (Y,{\rm O}(m))$; the fiber of $\bar{f}$ is diffeomorphic to a~compact
nilmanifold $N\slash \Gamma$ $($with $N$ being a~simply connected nilpotent Lie
group and $\Gamma\le N$ a co-compact lattice$)$, and the structure group is
contained in $(C(N)\slash (C(N)\cap \Gamma))\rtimes \operatorname{Aut}(\Gamma)$;
moreover, $\bar{f}$ induces a~singular fibration $f\colon M\to X=Y\slash {\rm O}(m)$ that
fits into the following commutative diagram:
\begin{align*}
\begin{split}
\begin{xy}
<0em,0em>*+{(F(M),{\rm O}(m))}="x",
<-7em, 0em>*+{N\slash \Gamma}="w",
<0em,-4em>*+{M}="v",
<9em,0em>*+{(Y,{\rm O}(m))}="y",
<9em,-4em>*+{X=Y\slash {\rm O}(m).}="u",
 "w";"x" **@{-} ?>*@{>}?<>(.5)*!/_0.5em/{ },
 "y";"u" **@{-} ?>*@{>} ?<>(.5)*!/_0.5em/{\scriptstyle \ \ \ \ \ \ \slash
 {\rm O}(m)},
 "v";"u" **@{-} ?>*@{>} ?<>(.5)*!/_0.5em/{\scriptstyle f},
 "x";"y" **@{-} ?>*@{>} ?<>(.5)*!/_0.5em/{\scriptstyle \bar{f}},
 "x";"v" **@{-} ?>*@{>} ?<>(.5)*!/_0.5em/{\scriptstyle \ \ \ \ \ \ \slash
 {\rm O}(m)},
\end{xy}
\end{split}
\end{align*}
Finally, $\bar{f}$ is an $\Psi(\varepsilon|m,D)$-Gromov--Hausdorff approximation
satisfying item~$(1)$ of Theorem~{\rm \ref{thm: Fukaya1}}, and the second fundamental
form of each~$\bar{f}$ fiber is uniformly bounded by $c(m)$ in magnitude.
\end{Theorem}
Here the key observation is that two isometries of a Riemannian manifold are
identical if their $1$-jets agree at some point, and the total space of
$1$-jets of isometries is conveniently represented by self-maps of the frame
bundle; see~\cite{Fukaya88}. Notice that the above mentioned fiber bundle
theorems are for manifolds collapsing also with bounded diameter, and they are
in the differentiable category. However, due to the existence of abundant local
symmetries for those very collapsed manifolds with bounded sectional curvature,
it is natural to wonder if the extra symmetry provided by the infranil fibers
can be reflected on the level of Riemannian metrics, locally around a given
fiber. This direction has been studied by Cheeger and Gromov~\cite{CGI, CGII}
for the central part of the infranil fibers (constructing the $F$-structure),
and is thoroughly investigated in the foundational work of Cheeger, Fukaya and
Gromov~\cite{CFG92}: on the very collapsed part of a complete Riemannian
manifold with bounded sectional curvature, a nilpotent Killing structure of
positive rank is constructed, providing the finest description of the
collapsing geometry.

We now define the \emph{nilpotent structure} (a.k.a.\ $N$-structure) on a given
complete Riemannian manifold $(M,g)$. Roughly speaking, it is the local singular
fiber bundle as in Theorem~\ref{thm: singular_fibration} patched together. Let
$\{U_j\}$ be a locally finite open covering of $M$, then for each $U_j$ we can
associate an \emph{elementary $N$-structure} $\mathfrak{N}_j$, which is nothing
but a singular fiber bundle $f_j\colon U_j\to X_j$ satisfying the description in
Theorem~\ref{thm: singular_fibration}~-- this can be seen as a localization of
that theorem, and all the information is encoded in the following
commutative diagram
\begin{align*}
\begin{split}
\begin{xy}
<0em,0em>*+{(F(U_j),{\rm O}(m))}="x",
<-7em, 0em>*+{N_j\slash \Gamma_j}="w",
<0em,-4em>*+{U_j}="v",
<9em,0em>*+{(Y_j,{\rm O}(m))}="y",
<9em,-4em>*+{X_j=Y_j\slash {\rm O}(m),}="u",
 "w";"x" **@{-} ?>*@{>}?<>(.5)*!/_0.5em/{ },
 "y";"u" **@{-} ?>*@{>} ?<>(.5)*!/_0.5em/{\scriptstyle \ \ {\rm pr}_j},
 "v";"u" **@{-} ?>*@{>} ?<>(.5)*!/_0.5em/{\scriptstyle f_j},
 "x";"y" **@{-} ?>*@{>} ?<>(.5)*!/_0.5em/{\scriptstyle \bar{f}_j},
 "x";"v" **@{-} ?>*@{>} ?<>(.5)*!/_0.5em/{\scriptstyle \ \ \overline{\rm pr}_j},
\end{xy}
\end{split}
\end{align*}
where $\overline{\rm pr}_j$ and ${\rm pr}_j$ denote the natural projections onto the space
of ${\rm O}(m)$ orbits. By the commutativity of this diagram, we see for any $x\in
X_j$ that $f_j^{-1}(x)=\overline{\rm pr}_j\big(\bar{f}_j^{-1}\big({\rm pr}_j^{-1}(x)
\big)\big)$ is an infranil manifold, called the
$\mathfrak{N}_j$-\emph{orbit} passing through any point of $f_j^{-1}(x)$; we
let $\mathcal{O}_j(p)$ denote such an orbit passing through a given $p\in
f_j^{-1}(x)$.
An open set $V\subset U_j$ is said to be $\mathfrak{N}_j$-\emph{invariant} if
it is the union of $\mathfrak{N}_j$-orbits of points in $V$. An $N$-structure
$\mathfrak{N}$ on $M$ is then a collection of elementary $N$-structure
$\{\mathfrak{N}_j\}$ satisfying the compatibility condition: there is an
ordering of $\{j\}$ such that if $U_{j}\cap U_{j'}\not=\varnothing$ with $j<j'$,
then it is both $\mathfrak{N}_j$- and $\mathfrak{N}_{j'}$-invariant; moreover,
there is an ${\rm O}(m)$-equivariant fiber bundle
$\bar{f}_{jj'}\colon \bar{f}_{j'}(F(U_j\cap U_{j'}))\to \bar{f}_{j}(F(U_j\cap
U_{j'}))$ such that $\bar{f}_{jj'}\circ \bar{f}_{j'}=\bar{f}_j$. Notice that
this implies $\mathcal{O}_{j'}(p)\subset \mathcal{O}_j(p)$ for any $p\in
U_j\cap U_{j'}$, and the $\mathfrak{N}$-orbit passing through $p\in M$ is then
defined as $\mathcal{O}(p):=\cup_j \mathcal{O}_j(p)$. Clearly, we have a~partition $M=\cup \mathcal{O}(p)$ into disjoint unions of $\mathfrak{N}$-orbits.
We also define the \emph{rank} of $\mathfrak{N}_j$ as $\rank
\mathfrak{N}_j:=\min\limits_{p\in U_j}\dim \mathcal{O}_j(p)$, and the \emph{rank} of
$\mathfrak{N}$ as $\rank \mathfrak{N}:=\min_j\rank \mathfrak{N}_j$. If
$N_j=N$ for all $j$, we say that $\mathfrak{N}$ is a~\emph{pure} nilpotent
structure, which is always the case when the collapsing sequence has uniformly
bounded diameter, as shown in~\cite{Fukaya89}. If each~$N_j$ is abelian, we say
that $\mathfrak{N}$ defines an $F$-structure.

Notice that on each $F(U_j)$, the fiber of $\bar{f}_j$ is diffeomorphic to the
symmetric space $N_j\slash \Gamma_j$, on which the simply connected nilpotent
Lie group $N_j$ acts (or equivalently, the sheaf of its Lie algebra maps
homomorphically into the sheaf of vector fields tangent to the $\bar{f}_j$
fibers). We say that the $N$-structure $\mathfrak{N}$ is a \emph{nilpotent
Killing structure} compatible with a Riemannian metric~$g'$ on~$M$, if the
canonically induced metric $\bar{g}'$ on $F(M)$ has its restriction
$\bar{g}'|_{F(U_j)}$ for each~$j$ being \emph{left} invariant under the actions
of $N_j$ (or equivalently, the sheaf of its Lie algebra maps homomorphically
into the sheaf of $\bar{g}'$-\emph{Killing} vector fields tangent to the
$\bar{f}_j$ fibers). Notice that the existence of the nilpotent Killing
structure makes each fiber bundle map $\bar{f}_j\colon F(U_j)\to Y_j$ a Riemannian
submersion. Moreover, the $\bar{g}'$-Killing vector fields induced by each
$N_j$ descends to $U_j$, defining $g'$-Killing vector fields tangent to the
$\mathfrak{N}_j$-orbits. When $\mathfrak{N}$ is pure, the partition
of $M$ into the $\mathfrak{N}$-orbits then defines a (singular) Riemannian
foliation, with leaves being the $\mathfrak{N}$-orbits; see~\cite{Molino88} and
Section~\ref{section5.2} for more discussions.

The main achievement of \cite{CFG92} can now be stated as following; see also
\cite[Theorem~5.1]{Rong07}.
\begin{Theorem}[nilpotent Killing structure, 1992]\label{thm: CFG}
There is a $\varepsilon(m)>0$ such that for any positive $\varepsilon
<\varepsilon(m)$, if $(M,g)$ is a complete Riemannian manifold with
$\sup_M |\mathbf{K}_g |\le 1$, and for any $x\in M$, $|B_g(x,1)|<\varepsilon$, then the following hold:
\begin{enumerate}\itemsep=0pt
 \item[$1)$] there exists a regular $(\rho,k)$-round metric
 $g_{\varepsilon}$ such that
 $ \|g-g_{\varepsilon} \|_{C^1(M,g)}<\Psi_{\rm CFG}(\varepsilon|m)$ holds
 for some $\Psi_{\rm CFG}(\varepsilon|m)>0$ with
 $\lim\limits_{\varepsilon \to 0}\Psi_{\rm CFG}(\varepsilon|m)=0$;
 \item[$2)$] there is a nilpotent Killing structure $\mathfrak{N}$ of positive rank,
 compatible with $g_{\varepsilon}$.
 \end{enumerate}
\end{Theorem}
Here we say that a Riemannian metric $g$ is $(\rho,k)$-\emph{round}, with $\rho$
and $k$ determined by $M$ and $\varepsilon$, if for any $p\in M$ there is
an open neighborhood $U\subset M$ of $B_g(p,\rho)$, such that $U$ has a~normal
Riemannian covering $\tilde{U}$ with deck transformation group $\Lambda$, whose
injectivity radius is uniformly bounded below by $\rho$, and there is an
isometric action on $\tilde{U}$ by a Lie group $N$ with nilpotent identity
component~$N^0$, extending the deck transformations by $\Lambda$ with index bound
$\big[N:N^0\big] = \big[\Lambda:\Lambda\cap N^0\big] \le k$.
In the presence of a compatible $N$-structure $\mathfrak{N}=\{\mathfrak{N}_j\}$,
for any $p\in M$ there is some $U_j$ such that $B_g(p,\rho)\Subset U_j$, and
$N^0\cong N_j\slash \mathbb{R}^d$ as Lie groups with $d=\dim G_{f_j(p)}$. Here
$f_j\colon U_j\to X_j$ is the local singular fibration over the collapsing limit
space~$X_j$, and~$G_{f_j(p)}$ is the isotropy group of $f_j(p)\in X_j$ as
discussed previously. Since the identity component of $G_{f_j(p)}$ is a
$d$-torus, its Lie algebra is indeed $\mathbb{R}^d$.
Moreover, we have $\Lambda\cap N_0\cong q(\Gamma_j)$ with $q\colon N_j\to N^0$ being
the natural quotient map.

Besides serving as a fundamental theorem of collapsing geometry with
bounded sectional curvature, Theorems~\ref{thm: singular_fibration} and~\ref{thm: CFG} have stimulated many exciting discoveries in Riemannian
geometry. To name a few, low dimensional collapsed manifolds were investigated,
for which some rationality conjectures of Cheeger and Gromov on the geometric
invariants associated to collapsing were verified~\cite{Rong93a, Rong95}, and
the $4$-dimensional case of Gromov's gap conjecture on the minimal volume was
confirmed \cite{Rong93b}; in \cite{CR95, CR96}, the singular structures
described in Theorem~\ref{thm: singular_fibration} were investigated, and the
existence of a mixed polarized sub-structure (or a pure polarization when there
is bounded covering geometry) was proven; in \cite{CCR01, CCR04}, an Abelian
structure has been constructed on very collapsed manifolds with bounded
non-positive sectional curvature, verifying the Buyulo conjecture
\cite{Buyalo90a, Buyalo90b} on $Cr$-structure in general dimensions; in
\cite{FR99, FR02, PRT99, PT99}, diffeomorphism stability and finiteness results
have been established for very collapsed $2$-connected manifolds; and in~\cite{CR09}, one-parameter families of collapsing metrics with bounded
sectional curvature have been constructed under the presence of the nilpotent
Killing structures of positive rank, extending the previous works of Cheeger
and Gromov~\cite{CGI}.

\subsection{Almost flatness by weaker curvature assumptions}\label{section2.2}
Besides the far-reaching generalization of Gromov's almost flat
manifold theorem to paramet\-ri\-zed versions in the collapsing geometry with
bounded sectional curvature, another direction of generalization is to weaken
the curvature assumption (\ref{eqn: AF_Rm}) in Theorem~\ref{thm: Gromov78}.

In \cite{DWY96}, Dai, Wei and Ye obtained a generalization for almost Ricci-flat
manifolds whose conjugate radii are uniformly bounded below.
 \begin{Theorem}[almost Ricci-flat manifolds, 1996] \label{thm: DWY96}
 There is a uniform constant $\varepsilon(m)\in (0,1)$ such that if a closed
 $m$-dimensional Riemannian manifold $(M,g)$ with conjugate radii bounded below
 by $1$ satisfies
\begin{align*}
 \diam(M,g)^2\max_M |\Rc_g |_g \le \varepsilon^2,
\end{align*}
then $M$ is diffeomorphic to an infranil manifold.
\end{Theorem}
While this theorem seems to be expected directly from Gromov's almost flat
manifold theorem, the proof actually requires the Ricci flow smoothing technique
-- to the authors' knowledge, this is the first instance where such a technique
is employed for manifolds satisfying certain Ricci curvature bounds. This
theorem is further generalized by Petersen, Wei and Ye in \cite{PWY99}, where
the regularity of the metrics is characterized by the \emph{harmonic}
$C^{0,\alpha}$ \emph{norm} of the manifolds (for any $\alpha\in (0,1)$); see
\cite[Theorem~1.4]{PWY99}, which we will not restate in this note for the sake
of brevity.

A more striking conjecture of Gromov \cite{GLP} predicted that
even if only assuming almost non-negative Ricci curvature, when the first Betti
number of the manifold is equal to its dimension, then it has to be a flat
$m$-torus. This conjecture is an effective version of the Bochner technique~\cite{Bochner, BY} which asserts that a Ricci non-negatively curved manifold
with maximal first Betti number (never exceeding the dimension) must be a flat
torus, and is confirmed by Colding based on his volume continuity theorem~\cite{Colding97}:
\begin{Theorem}[Colding--Gromov gap theorem, 1997]\label{thm: CG97}
There is a uniform constant $\varepsilon(m)\in (0,1)$ such that if a closed
$m$-dimensional Riemannian manifold $(M,g)$ satisfies
\begin{align*}
\diam(M,g)^2\Rc_g \ge -\varepsilon^2 g\qquad \text{and}\qquad b_1(M) = m,
\end{align*}
then $M$ is diffeomorphic to a flat torus.
\end{Theorem}
This theorem can also be viewed as a quantitative version of the Cheeger--Gromoll
splitting theorem \cite{CG71} when we consider the universal covering space of
the given manifold.
\begin{Remark}
In \cite{Colding97}, Colding proved that $M$ is homotopic to the torus when
$\dim M=3$, and is homeomorphic to the torus when $\dim M>3$. The diffeomorphism
statement was later proven in the joint work of Cheeger and Colding using the
Reifenberg method; see \cite[Appendix~A]{ChCoI}.
\end{Remark}

Obviously, the weaker assumption of almost non-negative Ricci
curvature by itself is not enough to detect the infranil manifold structure, and
certain extra assumptions should be expected. The second-named author proposed
in 2014 to study Riemannian manifolds with \emph{locally bounded Ricci covering
geometry}, considering those manifolds with Ricci curvature bounded below and
non-collapsing local universal covering spaces. Much progress has been made
since then (see Section~\ref{section3}), including the following almost flat manifold theorem due
to the second-named author and his collaborators~\cite{HKRX18}.
 \begin{Theorem}[bounded Ricci covering geometry, 2018]\label{thm:
HKRX18A} Given $m,v>0$, there is a uniform constant $\varepsilon(m,v)\in (0,1)$
such that if a closed $m$-dimensional Riemannian manifold $(M,g)$
sa\-tisfies
\begin{gather*}
\diam(M,g)^2\Rc_g \ge -\varepsilon^2 g\qquad\text{and}\qquad
 |B_{\tilde{g}}(\tilde{p}, 1) | \ge v,
\end{gather*}
then $M$ is diffeomorphic to an infranil manifold. Here $\tilde{p}$ is any
point in $\tilde{M}$, the universal covering space of~$M$ equipped
with the covering metric~$\tilde{g}$.
\end{Theorem}

Other types of weaker curvature assumptions include certain mixed curvature
conditions considered by Kapovitch: in \cite{Kapovitch19}, the sectional
curvature lower bound in Gromov's almost flat manifold theorem is weakened to a
Bakry--\'Emery Ricci tensor lower bound.
\begin{Theorem}[mixed curvature conditions, 2019]\label{thm: Kapovitch19}
There exists a uniform constant $\varepsilon(N) \in (0,1)$ such that if an
$m$-dimensional $(m\le N)$ weighted closed Riemannian manifold $\big(M,g,
{\rm e}^{-f}\mathcal{H}^m\big)$ satisfies
\begin{align*}
\diam(M,g)^2\max_{\wedge^2TM}\mathbf{K}_g \le \varepsilon(N)^2\qquad
\text{and}\qquad \diam(M,g)^2 (\Rc_g+\operatorname{Hess}_f ) \ge -\varepsilon(N)^2g,
\end{align*}
then $M$ is diffeomorphic to an infranil manifold.
\end{Theorem}
We notice here that the lower bound of the Bakry--\'Emery Ricci tensor
$\Rc_g+\operatorname{Hess}_f$ is essentially weaker than the corresponding Ricci curvature
lower bound, which would imply almost non-negative sectional curvature in the
context: as shown in \cite[Lemma~8.1]{Kapovitch19}, on the $2$-torus there is a
sequence of Riemannian metrics satisfying the assumptions of the theorem, but
the minimum of their sectional curvature (in $\wedge^2T\mathbb{T}^2$) have no
finite lower bound.
We also point out that it is more natural to consider the Bakry--\'Emery Ricci
curvature lower bound in the collapsing setting: see the work of Lott~\cite{Lott03}.

\subsection{Collapsing manifolds with Ricci curvature bounded below}\label{section2.3}
More generally, one may consider the Gromov--Hausdorff limits of manifolds in
the collec\-tion~$\mathcal{M}_{\rm Rc}(m)$ of complete Riemannian $m$-manfiolds with
the lowest eigenvalue of the Ricci tensor uniformly bounded below by $-(m-1)$.
While the Gromov--Hausdorff limits of a sequence in $\mathcal{M}_{\rm Rc}(m)$ with an
extra uniform volume lower bound (a \emph{non-collapsing sequence}) has been
thoroughly investigated through the works \cite{ChCoI, CJN18}, our understanding
of a possibly collapsing sequence in $\mathcal{M}_{\rm Rc}(m)$ (i.e., a sequence
\emph{without} volume lower bound) is very limited. The ideal here is to
develop a parallel theory as the collapsing geometry with bounded
sectional curvature, and describe the geometry around points where the
sectional curvature becomes unbounded.

Recall that in the classical theory, the collapsing with bounded sectional
curvature is caused by the extra symmetry of the infranil fibers. It is
therefore natural to focus on the local isometries of manifolds in
$\mathcal{M}_{\rm Rc}(m)$. The understanding of such
local isometry is encoded in the \emph{fibered fundamental group}
$\Gamma_{\delta}(p)$, defined for any $p\in M$ and $\delta\in (0,1)$ as
\begin{align*}
\Gamma_{\delta}(p) := \operatorname{Image} [\pi_1(B_g(p,\delta),p)\to \pi_1(B_g(p,2),p) ].
\end{align*}
This group collects all loops contained in
$B_g(p,\delta)$ and based at $p$, but that are allowed to deform within
$B_g(p,2)$. Concerning the structure of such groups, a key conjecture due to
Gromov states that the fibered fundamental group is almost nilpotent. This
conjecture is confirmed by Kapovitch and Wilking in \cite{KW11}.
\begin{Theorem}[generalized Margulis lemma, 2011]\label{thm: KW11}
There are uniform constants $C(m)\ge 1$ and $\varepsilon(m)\in (0,1)$ such that
for any $(M,g)\in \mathcal{M}_{\rm Rc}(m)$ and any $p\in M$, the fibered
fundamental group~$\Gamma_{\varepsilon(m)}(p)$ contains a nilpotent sub-group
of nilpotency rank $\le m$ and index $\le C(m)$.
\end{Theorem}
Based on a rescaling and contradiction argument, this theorem is strengthened by
Naber and Zhang \cite{NaberZhang} when a geodesic ball is Gromov--Hausdorff close
to a (lower) $k$-dimensional Euclidean $r$-ball (denoted by $\mathbb{B}^k(r)$
for any $r>0$): if $d_{\rm GH}\left(B_g(p,2), \mathbb{B}^k(2)\right)
<\varepsilon(m)$ for some $k\le m$, then $\rank \Gamma_{\varepsilon(m)}(p)\le
m-k$. More significantly, they discovered certain topological conditions that
guarantee a very strong regularity of the metric.
\begin{Theorem}[$\varepsilon$-regularity for Ricci curvature, 2018] \label{thm: NZ18}
For any $\varepsilon>0$ there is a uniform constant $\delta(m,\varepsilon)\in
(0,1)$ such that if $(M,g,p)$ is a pointed Riemannian $m$-manifold satisfying
$\Rc_g\ge -(m-1)g$ and $B_g(p,2)\Subset B_g(p,4)$, then for any normal
covering $\pi\colon (W,\hat{p})\to (B_g(p,2),p)$ with $\pi(\hat{p})=p$, covering
metric $\hat{g}$ and deck transformation group $G$, if
\begin{enumerate}\itemsep=0pt
 \item[$1)$] $d_{\rm GH}\big(B_g(p,2),\mathbb{B}^k(2)\big)<\delta$, and
 \item[$2)$] the almost nilpotent group $\widehat{G}_{\delta}(p):= \langle
 \gamma\in G\colon d_{\hat{g}}(\gamma.\hat{p},\hat{p})<2\delta \rangle$
 satisfies \begin{align*}\rank \widehat{G}_{\delta}(p) = m-k,\end{align*}
\end{enumerate}
 then for some $r\in (\delta,1)$ it holds that
\begin{align*}
d_{\rm GH}\big(B_{\hat{g}}(\hat{p},r), \mathbb{B}^m(r)\big) < \varepsilon r.
\end{align*}
 \end{Theorem}
If we impose the extra assumption of a uniform Ricci curvature upper bound, this
theorem directly proves a local fiber bundle theorem for domains collapsing to
lower dimensional Euclidean balls; see \cite[Proposition~6.6]{NaberZhang}. The
same conclusion actually holds without the Ricci curvature upper bound, and this
has been proven very recently by the authors based on the Ricci flow local
smoothing techniques \cite[Theorem~1.4]{HW20b}; see Theorem~\ref{thm:
local_infranil} in the next section. Notice that the assumption~(1) in
Theorem~\ref{thm: NZ18} is crucial in the original blow up argument, and the
case of orbifold collapsing limit in \cite[Theorem 1.4]{HW20b} is considerably
more difficult; see Remark~\ref{rmk: orbifold_fibration}. Based on these
results, the following theorem is recently proven in \cite{HW20a}.
\begin{Theorem}[rigidity of the first Betti number, 2020]\label{thm: HW20a}
Given $D\ge 1$ and $v>0$, there is a~uniform constant $\varepsilon(m,D,v)\in
(0,1)$ such that if $(M,g)\in \mathcal{M}_{\rm Rc}(m)$ and $(N,h)\in
\mathcal{M}_{\rm Rm}(k,D,v)$ satisfy $d_{\rm GH}(M,N)<\varepsilon(m,D,v)$ with $k\le
m$, then $b_1(M)-b_1(N)\le m-k$. Moreover, if the equality holds, then $M$ is
diffeomorphic to an $(m-k)$-torus bundle over~$N$.
\end{Theorem}
Here we notice that the assumptions of Theorem~\ref{thm: NZ18} are purely local,
and a key difficulty in the proof was to localize the topological information
encoded in the first Betti number, which is global in nature. Here we
introduced the so-called \emph{pseudo-local fundamental group}
$\tilde{\Gamma}_{\delta}(p)$, which is defined for any $p\in M$ and $\delta\in
(0,1)$ as $\tilde{\Gamma}_{\delta}(p) :=\operatorname{Image}\left[\pi_1(B_g(p,\delta),p)\to
\pi_1(M,p)\right]$. This concept provides a bridge linking $\pi_1(M,p)$ with
$\Gamma_{\delta}(p)$. We also considered $H_1^{\delta}(M; \mathbb{Z})$,
generated by singular homology classes with a representation by a geodesic loop
of length not exceeding $10\delta$. Under the assumption $d_{\rm GH}(M,N)<\delta$,
it is then shown that $b_1(M)-b_1(N)=\rank H_1^{\delta}(M;\mathbb{Z})$; and
for $\delta<\varepsilon(m)$ sufficiently small, generalizing the work of
Colding and Naber \cite{CN12} it is shown that $\rank H_1^{\delta}
(M;\mathbb{Z})\le \rank \tilde{\Gamma}_{\varepsilon(m)}(p)\le m-k$ for any
$p\in M$. Therefore, by Theorem~\ref{thm: NZ18} the assumption
$b_1(M)-b_1(N)=m-k$ ensures the universal covering of $M$ to locally resemble
the $m$-Euclidean space. We could then run the Ricci flow to obtain a regular
metric that still collapses, and applying Theorem~\ref{thm: Fukaya1} we
established the theorem~-- discussions on such Ricci flow smoothing technique
will be in Section~\ref{section4}.

\section{Locally bounded Ricci covering geometry}\label{section3}
In this section, we discuss the program initiated by the second-named author
around 2014 to investigate Riemannian manifolds with Ricci curvature bounded
below and non-collapsing local universal covering spaces: its current status
and its goal; see also the previous survey \cite{Rong18} by the second-named
author for related discussions.

More precisely, we let $\widetilde{\mathcal{M}}_{\rm Rc}(m,\rho,v)$ denote the
collection of complete $m$-dimensional Riemannian manifolds $(M,g)$ satisfying
$\Rc_g\ge -(m-1)g$ and for any $x\in M$, $|B_{\tilde{g}}(\tilde{x}, \rho)
|\ge v$, where~$\tilde{g}$ is the covering metric of the (incomplete)
Riemannian universal covering space of $B_g(x,\rho)$, and~$\tilde{x}$ is any
point covering~$x$. We call the quantity $|B_{\tilde{g}}(\tilde{x},
\rho)|$ the \emph{local rewinding volume} of~$x$ at scale~$\rho$, denoted
by $\widetilde{\rm Vol}_g(x,\rho)$. So roughly speaking,
$\widetilde{\mathcal{M}}_{\rm Rc}(m,\rho,v)$ consists of manifolds with Ricci
curvature and local rewinding volume (at scale~$\rho$) uniformly bounded below,
and we say that such manifolds have \emph{locally} $(\rho,v)$\emph{-bounded
Ricci covering geometry}, or just locally bounded Ricci covering geometry.

The goal of studying locally bounded Ricci covering geometry is mainly to
establish an analogue, for manifolds in $\widetilde{\mathcal{M}}_{\rm Rc}(m,\rho,v)$
and their (pointed) Gromov--Hausdorff limits, of the nilpotent structure theorey
of Cheeger, Fukaya and Gromov (see Section~\ref{section2.1}).

 As pointed out in the introduction, understanding the collapsing behaviors of
 manifolds of this type is a very natural and immediate step in our study of
 general collapsing phenomena of manifolds with Ricci curvature bounded below.
 If a Riemannian $m$-manifold has sectional curvature uniformly bounded by~$1$
 in absolute value, then it has locally $(\rho(m),v(m))$-bounded Ricci covering
 geometry, where the constants only depend on the dimension~$m$; see~
 \cite{CFG92, Rong20}.

\subsection{Singular infranil fiber bundles}\label{section3.1}
Recall that our ideal of study collapsing geometry with Ricci curvature bounded
below is to recover, at least over most parts of the collapsing limit, the
infranil fiber bundle structure. In the setting of collapsing with locally
bounded Ricci covering geometry, this is indeed the case over the regular part
of the collapsing limit.
\begin{Theorem}[infranil fiber bundle, 2020] \label{thm: HR20}
Given $m,k\in \mathbb{N}$ with $k<m$, $D\ge 1$ and $\rho,v_1,v_2>0$ there is a
uniform constant $\delta(m,D,\rho, v_1,v_2)>0$ such that if $(M,g)\in
\widetilde{\mathcal{M}}_{\rm Rc}(m,\rho,v_1)$ is $\delta$-Gromov--Hausdorff close to
a manifold $(N,h)\in \mathcal{M}_{\rm Rm}(k,D,v_2)$, then there is a fiber bundle
$f\colon M\to N$ which is also a $\Psi(\delta|m,D,\rho,v_1,v_2)$-Gromov--Hausdorff
approximation, with whose fibers diffeomorphic to an $(m-k)$-dimensional
infranil manifold, and whose structure group reduced to a generalized torus
group as described in Theorem~{\rm \ref{thm: singular_fibration}}.
\end{Theorem}
This theorem summarizes the contributions from \cite{Huang20} and \cite{Rong20}.
In \cite{Huang20}, the existence of the topological fiber bundle map $f$ is
obtained using the canonical Reifenberg method in \cite{CJN18}. However, the
infranil manifold structure of the fiber obtained in \cite{Rong20} is not a
direct application of Theorem~\ref{thm: HKRX18A}, since an $f$-fiber may have no
uniform Ricci curvature lower bound. Considerations on the ambient geometry is
instead carried out in \cite{Rong20} -- this work, when restricting to the
case of almost flat manifolds, provides for the first time an approach entirely
different from the original one in \cite{Gromov78b, Ruh82}; see also
\cite{Rong19}. Notice that the arguments in \cite{Huang20, Rong20} are local
and our statements here are valid even if $N$ is an open set in an
$k$-dimensional manifold.

Besides substaintially generalizing the condition of collapsing with bounded
sectional curvature, the condition of locally bounded Ricci covering geometry is
also more general than the assumption of maximal nilpotency rank discussed in
Theorem~\ref{thm: NZ18}, which plays a key role in the proof of
Theorem~\ref{thm: HW20a}. In fact, even when collapsing with bounded sectional
curvature occurs, the fibered fundamental group (or the pseudo-local
fundamental group) based at the fiber over a corner point cannot have maximal
nilpotency rank. Here we recall that for a collapsing limit space $X$, we have a
singular fiber bundle $f\colon M\to X$ as described in Theorem~\ref{thm:
singular_fibration}, and we say $x\in X$ is a \emph{corner point} if $\dim
G_x>0$. We make, however, the following remark.
\begin{Remark}
When the collapsing limit $X$ is a manifold, these three concepts of collapsing
coincide: both (a) collapsing with locally bounded Ricci covering geometry, and
(b)~collapsing with Ricci curvature bounded below and maximal nilpotency rank at
every point, imply the same collapsing infranil fiber bundle structure arising
from collapsing with bounded sectional curvature, which clearly implies the cases~(a) and~(b).
\end{Remark}

For a general metric space arising as the Gromov--Hausdorff limits of
manifolds in \linebreak $\widetilde{\mathcal{M}}_{\rm Rc}(m,\rho,v)$, we propose the following
conjecture; compare also Theorem~\ref{thm: singular_fibration}.
\begin{Conjecture}[singular nilpotent fibration] Suppose a sequence
$ \{(M_i,g_i) \}\subset \widetilde{\mathcal{M}}_{\rm Rc}(m,\rho,v)$
collapses to a lower dimensional compact metric space~$X$, i.e.,
$M^m_i\xrightarrow{\rm GH} X^k$ where the dimension is the sense of Colding--Naber~{\rm \cite{CN12}}. Then for each~$i$ sufficiently large, there is a singular
fibration $f_i \colon M_i\to X$, such that a regular fiber is an infranil manifold,
and a singular fiber is a finite quotient of an infranil manifold.
\end{Conjecture}
Some progress on understanding the structure of a Gromov--Hausdorff limit of a
sequence in $\widetilde{\mathcal{M}}_{\rm Rc}(m,\rho,v)$ has already been made:
see, e.g., \cite{HR20}. Notice that if $ \{(M_i,g_i) \}\subset
\mathcal{M}_{\rm Rc}(m)$ and $M_i\xrightarrow{\rm GH}X$ for some metric space $(X,d)$
with $\diam (X,d)\le D$, then there is a renormalized measure $\nu$ on $X$, and
by the work of Colding and Naber \cite{CN12}, there is a unique $k\le m$ such
that at $\nu$-a.e.\ point of $X$ any tangent cone is isometric to
$\mathbb{R}^k$: this $k$ will be denoted as $\dim_{\rm CN}X$, the dimension of~$X$
in the sense of Colding--Naber. On the other hand, as a metric space one can
talk about the Hausdorff dimension $\dim_{\rm H}X$. When $k=m$ we know that
$\dim_{\rm CN}X=\dim_{\rm H}X$, but when $k<m$ this it remains an open question whether
$\dim_{\rm CN}X=\dim_{\rm H}X$: see, e.g., \cite[Open Question~1.11]{KL18}.
However, for manifolds with locally bounded Ricci covering
geometry, we have the following result~\cite{Rong18}.
\begin{Proposition}
If $ \{(M_i,g_i,p_i) \}\subset \widetilde{\mathcal{M}}_{\rm Rc}(m,\rho,v)$
and $M_i\xrightarrow{p{\rm GH}}X$ with $(X,d,p)$ a pointed metric space, then
there is some $k\le m$ such that any tangent cone at any point of~$X$ is
a~$k$-dimensional metric cone. Moreover, $\dim_{\rm H}X=k=\dim_{\rm CN}X$.
\end{Proposition}

\subsection{Almost maximal local rewinding volume}\label{section3.2}
In this sub-section, we consider an ``extremal'' case of locally bounded
Ricci covering geometry: those manifolds in
$\widetilde{\mathcal{M}}_{\rm Rc}(m,\rho,v)$ with almost maximal local rewinding
everywhere; compare also~\cite{Colding96}. Given $(M,g)\in
\widetilde{M}_{\rm Rc}(m,\rho,v)$, since for any $x\in M$, the local covering
metric has the same Ricci curvature lower bound as the original metric, the
Bishop--Gromov volume comparison is in effect, implying that
$\widetilde{\rm Vol}_g(\tilde{x},\rho)\le \Lambda_{\lambda(x)}^m(\rho)$, which
denotes the volume of a geodesic $\rho$-ball in the space form of sectional
curvature equal to $\lambda(x)$, with $\lambda(x)$ denoting the lowest
eigenvalue of~$\Rc_g$ on~$B_g(x,2\rho)$. When $M$ is compact, we let
$\lambda(M,g):=\min\limits_{x\in M}\lambda(x)$. If, however, we know that the local
rewinding volume is almost maximal, then strong structural results have been
obtained by the second-named author with his collaborators in~\cite{CRX18, CRX19}.
\begin{Theorem}[quantitative space form rigidity, 2019]\label{thm: CRX19a}
Given $m,\rho,v>0$ with $\rho<1$ and a closed manifold $(M,g) \in
\mathcal{M}_{\rm Rc}(m)$ whose Riemannian universal covering $(\tilde{M},
\tilde{g})$ has some $\tilde{p}_0\in \tilde{M}$ satisfying $|B_{\tilde{g}}
(\tilde{p}_0,1) |_{\tilde{g}}\ge v$, then we have the following:
\begin{enumerate}\itemsep=0pt
 \item[$1)$] if $\lambda(M,g)=1$ and $\inf\limits_{x\in M}
 \widetilde{\rm Vol}_g(x,\rho)\ge (1-\varepsilon)\Lambda_{1}^m(\rho)$ for some
 uniform $\varepsilon\in (0,1)$ determined by $m$, $\rho$ and $v$, then $M$ is
 diffeomorphic to a spherical space form by a
 $\Psi(\varepsilon|m,\rho,v)$-isometry;
 \item[$2)$] if $\lambda(M,g)=0$, $\diam (M,g)\le 1$ and $\inf\limits_{x\in M}
 \widetilde{\rm Vol}_g(x,\rho)\ge (1-\varepsilon)\Lambda_0^m(\rho)$ for some
 uniform $\varepsilon\in (0,1)$ determined by $m$, $\rho$ and $v$, then $M$ is
 isometric to a flat manifold; and
 \item[$3)$] if $\lambda(M,g)=-1$, $\diam (M,g)\le d$ and $\inf\limits_{x\in M}
 \widetilde{\rm Vol}_g(x,\rho)\ge (1-\varepsilon)\Lambda_{-1}^m(\rho)$ for some
 uniform $\varepsilon\in (0,1)$ determined by~$d$, $m$, $\rho$ and~$v$, then
 $M$ is diffeomorphic to a hyperbolic manifold by a~$\Psi(\varepsilon|m,\rho,v,d)$-isometry.
\end{enumerate}
\end{Theorem}
Note that manifolds satisfying items~(1) or~(2) in this theorem may be
arbitrarily collapsed. In \cite[Theorem~D]{CRX19}, a quantitative rigidity
theorem for hyperbolic spaces (compare item~(3) in Theorem~\ref{thm: CRX19a})
has been obtained for manifolds in $\mathcal{M}_{\rm Rc}(m)$ in terms of the volume
entropy~\cite{LW10}. In fact, by the assumed uniform lower bound of the local
rewinding volume, it is natural to ask if we can drop the non-collapsing
assumption of the universal covering spaces:
\begin{Conjecture}
Theorem~{\rm \ref{thm: CRX19a}} still holds for $(M,g)\in \mathcal{M}_{\rm Rc}(m)$ even if
we do not assume the existence of $\tilde{p}_0\in \tilde{M}$ so that
$|B_{\tilde{g}} (\tilde{p}_0,1)|_{\tilde{g}}\ge v>0$.
\end{Conjecture}
In the case when a uniform Ricci curvature upper bound is additionally assumed,
this conjecture has been verified by the second-named author and his
collaborators in~\cite{CRX18}.

\section[Smoothing the locally collapsing metrics with Ricci curvature bounded
below]{Smoothing the locally collapsing metrics\\ with Ricci curvature bounded
below}\label{section4}

As mentioned in the introduction, the proofs of Theorems \ref{thm: Fukaya1},
\ref{thm: HKRX18A}, \ref{thm: Kapovitch19} and \ref{thm: HW20a} all rely on the
smoothing effect by globally running the Ricci flow. The Ricci flow with
initial data $(M,g)$, first introduced by Hamilton \cite{Hamilton} on closed
$3$-manifolds to deform a given Riemannian metric $g$ with positive Ricci
curvature to a positive Einstein metric, is a smooth family of Riemannian
metrics~$g(t)$ on~$M$ solving the following initial value problem for $t\ge 0$:
\begin{align}\label{eqn: RF_defn}
\begin{cases}
\partial_tg(t) = -2\Rc_{g(t)},\\
g(0) = g.
\end{cases}
\end{align}
In harmonic coordinates, the Ricci flow becomes a non-linear
heat-type equation for the metric tensor, and by the nature of the heat flows,
notably Shi's estimates~\cite{Shi}, a key effect of running Ricci flow is that
the evolved metric has much improved regularity:
\begin{align}\label{eqn: Shi_estimate}
\forall\, l\in \mathbb{N},\ \exists\, C_l>0,\ \sup_M\big|\nabla^l
\Rm_{g(t)}\big|_{g(t)} \le C_lt^{-1-l}.
\end{align}
Here the constants $C_l$ depend on the dimension of $M$, as well as
$\|\Rm_g\|_{C^0(M,g)}$. In fact, the finiteness of
$\|\Rm_g\|_{C^0(M,g)}$ guarantees the Ricci flow solution to~(\ref{eqn: RF_defn}) to exist for a definite amount of time determined by
its value, even if~$(M,g)$ is complete but non-compact.

In view of Shi's estimates, the Ricci flow also becomes a useful tool to smooth
a given Riemannian metric by replacing the initially given metric $g$ with the
evolved metric $g(t)$, whose regularity is controlled by~(\ref{eqn:
Shi_estimate}) -- in order to take advantage of such an estimate, a uniform
lower bound of the Ricci flow existence time then becomes crucial. This method
has been investigated in \cite{DWY96} for closed mainfolds, producing fruitful
applications, such as the proofs of Theorem~\ref{thm: DWY96} and
\cite[Proposition 6.6]{NaberZhang} (see also \cite{PWY99}); but as the
collapsing phenomenon may be observed locally on a geodesic ball, the
localization of the Ricci flow existence results is usually necessary for the
smoothing purpose. In this section, we will discuss the recent developments on
the Ricci flow local smoothing techniques for collapsing initial data with
Ricci curvature bounnded below.

\subsection{Local existence of the Ricci flow}\label{section4.1}
As shown in \cite[Lemma 2.2]{HW20b}, we could in fact start the Ricci flow
locally on any Riemannian manifold with Ricci curvature bounded below.
\begin{Lemma}\label{lem: HW20b0}
Given a complete Riemannian manifold $(M^m,g)$ with $\Rc_g\ge -(m-1)g$ and let~$K$ be a compact subset. For any $R>0$ there is a smooth family of Riemannian
metrics $g(t)$ on $B_g\big(K,\frac{R}{4}\big)$ satisfying
\begin{align*}
\begin{cases}
\partial_tg(t) = -2\Rc_{g(t)}\ \text{on}\ B_g\big(K,\frac{R}{4}\big)\times (0,T],\\
g(0) = g\ \text{on}\ B_g\big(K,\frac{R}{4}\big),
\end{cases}
\end{align*}
for some $T>0$, such that
\begin{align*}
\forall\, t\in (0,T],\quad
\sup_{B_g(K,\frac{R}{4})} |\Rm_{g(t)} |_{g(t)} \le Ct^{-1},
\end{align*}
where the positive constants $C$ and $T$ depend on $g$, $K$ and $R$.
\end{Lemma}
This lemma is proven using Hochard's conformal transformation technique in
\cite[Section~6]{Hochard16}. One can conformally blow the points near $\partial
B_g(K,R)$ to infinity, obtaining a complete Riemannian metric~$h$ defined on
$B_g(K,R)$. Since $\overline{B_g(K,R)}$ is compact, the sectional curvature of
$g$ is bounded, and thus so is the sectional curvature of the complete metric
$h$ by making a good choice of the conformal factor. One could then rely on
Shi's short time existence theorem to start a Ricci flow solution $h(t)$ with
initial data $h(0)=h$. But since the conformal factor can be designed to be~$1$
on $B_g\big(K,\frac{R}{4}\big)$, one can view the Ricci flow solution
$h(t)|_{B_g(K,\frac{R}{4})}$ as the local Ricci flow starting from the initial
data $\big(B_g\big(K,\frac{R}{4}\big),g\big)$. We point out that the conformal factor could
also be designed so that the scalar curvature and local isoperimetric constant
lower bounds for $h$ are comparable to the ones for~$g$.

As mentioned in the introduction, in order to use Ricci flow as a smoothing
tool one needs a definite lower bound of the existence time. If the initial
data has only Ricci curvature lower bound, then the existence time lower bound
relies on certain non-collapsing condition -- even if the actual
$m$-dimensional initial data may collapse to a lower dimensional space, the
local covering spaces are usually assumed to resemble the local $m$-dimensional
Euclidean space. In this setting, our most recent result \cite[Theorem~1.2]{HW20b} gives:
 \begin{Theorem}\label{thm: HW20b1}
Given any $\alpha\in \big(0,10^{-1}\big)$, any positive $m, l\in \mathbb{N}$ and any
$R\in (0,100)$, there are uniform constants $\delta_{O}(m,l,R,\alpha),
\varepsilon_{O}(m,\alpha)\in (0,1)$ to the following effect: let $K$ be a
compact and connected subset of $(M^m,g)$, an $m$-dimensional Riemannian
manifold with $\Rc_g\ge -(m-1)g$, suppose for some $k\le m$ and $\delta\le
\delta_O$ it satisfies for any $p\in B_g(K,R)$ the following assumptions:
\begin{enumerate}\itemsep=0pt
 \item[$1)$] there are a finite group $G_p<{\rm O}(k)$ with $ |G_p |\le l$ and a
 $\phi_p \in \operatorname{Hom} (\pi_1(B_g(K,R),p),G_p )$ which is
 surjective,
 \item[$2)$] $d_{\rm GH}\big( B_g\big(p,4^{-1}R\big), \mathbb{B}^k\big(4^{-1}R\big)\slash
 G_p\big)<\delta$, and
 \item[$3)$] $\rank \tilde{\Gamma}_{\delta}(p)=m-k$,
 \end{enumerate}
then there is a Ricci flow solution with initial data $\big(B_g\big(K,\frac{R}{4}\big),g\big)$,
existing for a period no shorter than~$\varepsilon_{O}^2$, and with curvature
control
\begin{align*}
\forall\, t\in \big(0,\varepsilon_{O}^2\big],\quad
\sup_{B_g(K,\frac{R}{4})} |\Rm_{g(t)} |_{g(t)} \le \alpha t^{-1} +\varepsilon_{O}^{-2}.
\end{align*}
\end{Theorem}
Here the notation
$\tilde{\Gamma}_{\delta}(p):=\operatorname{Image} [\pi_1(B_g(p,\delta),p)\to
\pi_1(B_g(K,R),p) ]$ is the pseudo-local fundamental
group for $B_g(K,R)$, containing all geodesic loops in $B_g(p,\delta)$ with base
point~$p$, and are allowed to be deformed within the entire $B_g(K,R)$.

Theorem~\ref{thm: HW20b1} is proven roughly as following: by
conditions~(1) and~(2), for each $p\in B_g\big(K,\frac{3}{4}R\big)$ we can find a finite
normal covering of~$B_g\big(p,\frac{R}{4}\big)$, so that it is
$\Psi(\delta)$-Gromov--Hausdorff close to~$\mathbb{B}^k\big(\frac{R}{4}\big)$; this
condition, together with the nilpotency rank assumption in condition~(3),
enable us to show that the isoperimetric constant in a fix-sized geodesic ball
around any point of the universal covering space of $B_g\big(K,\frac{3}{4}R\big)$ is
very close to the $m$-Euclidean isoperimetric constant. By the design of the
conformal factor, such almost locally Eucliean property is almost
preserved under the conformal transformation, and together with the (relaxed)
scalar curvature lower bound of the conformally transformed metric, it enables
us to apply Perelman's pseudo-locality theorem (see the next sub-section) to bound
the Ricci flow existence time from below.

Theorem~\ref{thm: HW20b1} characterizes the ``almost locally Euclidean
covering space'' assumption via algebraic conditions, i.e., the maximality of
the rank of the pseudo-local fundamental groups and the existence of a
surjective homomorphism of the local fundamental group onto the orbifold
groups. One could also directly assume that the local universal covering space
resembles the $m$-Euclidean space up to a fixed scale, defining the so-called
$(\delta,\rho)$-Reifenberg points. For any $p\in M^m$, we say it is a
$(\delta,\rho)$\emph{-Reifenberg point}, if for any lift $\tilde{p}$ of $p$ in
the Riemannian universal covering space of $B_g(p,\rho)$,
\begin{align*}
\forall\, r\in (0,\rho],\quad r^{-1}d_{\rm GH} \big(B_{\tilde{g}}(\tilde{p},r),
\mathbb{B}^m(r) \big) \le \delta.
\end{align*}
This definition essentially appears in the work~\cite{HKRX18} of the
second-named author and his collaborators, and is for the purpose of defining
the concept of Ricci bounded local covering geometry. Notice that with bounded
Ricci curvature, if~$p$ is a~$(\delta,2\rho)$-Reifenberg point for~$\delta$
sufficiently small, then $B_{\tilde{g}}(\tilde{p},\rho)$ has a uniform lower
bound on the $C^{1,\frac{1}{2}}$ harmonic radius. On the other hand, one could
always run the Ricci flow locally around a $(\delta,\rho)$-Reifenberg point for
a definite amount of time.
\begin{Theorem}\label{thm: HKRX18b}
For any $\alpha,\rho \in (0,1)$ there are uniform $\delta(m,\alpha,\rho) \in
(0,1)$ and $\varepsilon(m,\alpha,\rho)\in (0,1)$ such that if $(M,g)$ is a
complete Riemannian manifold with $\Rc_g\ge-(m-1)g$, and $p\in M$ is a
$(\delta,2\rho)$-Reifenberg point, then there is a Ricci flow solution with
initial data $(B_g(p,\rho),g)$, that exists up to time $\varepsilon^2$
and for any $t\in \big[0,\varepsilon^2\big]$, the curvature satisfies
\begin{align*}
\sup_{B_g(p,\rho)} |\Rm_{g(t)} | \le \alpha t^{-1}.
\end{align*}
\end{Theorem}
\begin{proof}[Sketch of proof]
We could always start the Ricci flow $h(t)$ by Lemma~\ref{lem: HW20b0} on
$B_g(p,2\rho)$. Moreover, we could make sure that the initial data satisfies
$h(0)|_{B_g(p,\rho)}\equiv g|_{B_g(p,\rho)}$. We only need to bound the
existence time of the Ricci flow from below, which in turn relies on showing
that the isoperimetric constant at any point of the covering space is almost
Euclidean on a fixed scale. This is proven in \cite{CM17}, thanks to the
definition of the $(\delta,\rho)$-Reifenberg property, as long as $\delta$ is
sufficiently small. One then relies on Perelman's pseudo-locality theorem to
prove that the flow exists for a definite period of time.
\end{proof}

Let us also mention that in the case of Kapovitch's mixed curvature condition,
it can be shown that the assumptions of Theorem~\ref{thm: Kapovitch19}
guarantees the local universal covering space at a point of the manifold to be
 almost locally Euclidean, via the aspherical theorem~\cite[Theorem~5.3]{Kapovitch19}.

\subsection{The pseudo-locality theorem}\label{section4.2}
In all the results discussed above, once the almost locally Euclidean condition
for the local covering space is verified, the lower bound of the existence time
of the Ricci flow is guaranteed by Perelman's pseudo-locality theorem, stated
in its various forms as following:
\begin{Theorem}[Perelman's pseudo-locality theorem]
For any $\alpha\in (0,1)$, there are uniform positive constants
$\varepsilon_P=\varepsilon_P(m,\alpha)$ and $\delta_P=\delta_P(m,\alpha)$
such that if $(M,g)$ is a Ricci flow solution define for $t\in [0,T]$ with each
time slice $(M,g(t))$ being a complete Riemannian manifold, and if one of the
conditions holds for $p\in M$:
\begin{enumerate}\itemsep=0pt
 \item[$1)$] $\Sc_{g(0)}\ge -1$ on $B_{g(0)}(p,1)$ and $I_{B_{g(0)}(p,1)}\ge
 (1-\delta_P)I_m$, or
 \item[$2)$] $\Rc_{g(0)}\ge -\delta_P g(0)$ on $B_{g(0)}(p,1)$ and
 $|B_{g(0)}(p,1)|_{g(0)}\ge (1-\delta_P)\omega_m$,
\end{enumerate}
where $I_m$ and $\omega_m$ stands for the isoperimetric constant and volume of
the $m$-Euclidean unit ball, respectively, and $I_{\Omega}$ denotes the
isoperimetric constant for the domain $\Omega\subset M$, then
\begin{align}\label{eqn: pseudolocality_Rm}
\forall\, t\in \big(0,\varepsilon_P^2\big],\quad
\sup_{B_{g(t)}(p,\varepsilon_P)} |\Rm_{g(t)} |_{g(t)} \le \alpha
t^{-1} +\varepsilon_P^{-2}.
\end{align}
\end{Theorem}
The theorem originates from Perelman's work for closed manifolds satisfying
condition~(1) above; see \cite[Theorem 10.1]{Perelman}. Later a version for
complete non-compact manifolds was obtained by Chau, Tam and Yu; see
\cite[Theorem~8.1]{CTY11}. The theorem with condition~(2) was proven by Tian
and the third-named author for closed manifolds in \cite[Proposition~3.1]{TW15}, and its counterpart for complete non-compact data appears in
the recent work of the authors' in \cite[Proposition~6.1]{HW20a}. We point out
that all the later works essentially follow Perelman's original idea and
arguments.

In the proofs of Theorems~\ref{thm: HW20b1},~\ref{thm: HKRX18b} and
\cite[Theorem~7.2]{Kapovitch19}, the almost locally Euclidean property for the
local covering spaces checked before allows us to apply the pseudo-locality
theorem to the covering flow and obtain a uniform lower bound on the existence
time of the Ricci flow started via Lemma~\ref{lem: HW20b0}: if the existence
time~$T$ of the Ricci flow were shorter than $\varepsilon_P^2$, then for some
sequence $t_i\nearrow T$ we could observe points $x_i\in M$ such that
$\lim\limits_{t_i\to T} |\Rm_{h(t_i)} |_{h(t_i)}(x_i)= \infty$; especially,
we will get $ |\Rm_{h(t_i)} |_{h(t_i)}(x_i)>2\alpha
T^{-1}+\varepsilon_P^{-2}$ for all $i$ large enough, contradicting the
conclusion (\ref{eqn: pseudolocality_Rm}) since $T>0$ is fixed.

Heuristically speaking, Perelman's pseudo-locality theorem tells that the Ricci
flow locally ``preserves'' the almost Euclidean parts of the manifolds. And it
is natural to wonder if the initial data locally approaches lower dimensional
Euclidean spaces, whether a pseudo-locality type theorem still holds. In fact,
after proving a version of the pseudo-locality theorem \cite[Theorem~10.3]{Perelman}, Perelman asked:

\emph{``A natural question is whether the assumption on the volume of the ball
is superfluous.''}

We notice however, that there are examples (see, e.g., \cite{Lu10, HKRX18}) that
show the direct removal of the initial local volume non-collapsing assumption
is fatal:
\begin{Example}[Topping]
Let $M_{\delta}$ denote the smooth manifold obtained from capping off the
$\delta$-thin cylinder $\delta \mathbb{S}^1\times [-1,1]$ ($\mathbb{S}^1$ is
identified with the unit circle in $\mathbb{C}$ with base point $1\in
\mathbb{C}$) by two discs of radius approximately~$\frac{\pi}{2}\delta$ and
slightly smoothing near the ends of the cylinder. The natural metric~$g_{\delta}$ can be easily made to have non-negative sectional curvature. It is
also obvious that around the base point $p_{\delta}=(\delta,0)$ of $M_{\delta}$,
the geodesic ball $B_{g_{\delta}}\big(p_{\delta},\frac{1}{2}\big)$ is flat and is
$\delta$-Gromov--Hausdorff close to $\big({-}\frac{1}{2},\frac{1}{2}\big)$.
 However, the Ricci flow starting from~$M_{\delta}$ exits only for a period
 determined by the area of $M_{\delta}$, which is proportional to~$\delta$.
 Therefore, as $\delta\searrow 0$, a~curvature bound of the form~(\ref{eqn:
 pseudolocality_Rm}) cannot be obtained for any uniform $\varepsilon>0$.
\end{Example}

Fortunately, in many natural settings, the scalar curvature is indeed uniformly
bounded along the Ricci flow, and here we raise the following
\begin{Conjecture}
Given $\alpha\in (0,1)$, there are positive constants
$\delta=\delta(m,\alpha)$ and $\varepsilon=\varepsilon(m,\alpha)$ such that
if $(M,g)$ is an $m$-dimensional Ricci flow solution on $[0,T]$ with each of whose
time slices being complete, and for some $p\in M$ it satisfies
\begin{align*}
\sup_{B_{g(0)}(p,1)} |\Rm_{g(0)} |_{g(0)} \le 1,\qquad
d_{\rm GH}\big(B_{g(0)}(p,1),\mathbb{B}^k(1)\big) \le
\delta,\qquad \sup_{M\times [0,T]} |\Sc_{g(t)} | \le 1,
\end{align*}
then we have for any $t\in \big[0,\varepsilon^2\big]$, the curvature bound
\begin{align*}
\sup_{B_{g(t)}(p,\varepsilon)} |\Rm_{g(t)} |_{g(t)} \le \alpha
t^{-1} + \varepsilon^{-2}.
\end{align*}
\end{Conjecture}

\subsection{Distance distortion estimates}\label{section4.3}
Once the Ricci flow exists for a definite amount of time, for the purpose of
smoothing, it is of key importance to compare the initial metric with the
evolved metric. In general, the distance distortion estimate for Ricci flows
is of central importance in the understanding of the geometry along the Ricci
flows, and we refer the readers to \cite{BZ17, BZ19,CW12,CW19,CW14,H95,Foxy1808, TW15} for previous works on this topic in various settings.
Very recently, based on the previous contributions, especially the local
entropy theory developed in \cite{Wang18}, the distance distortion estimate for
collapsing initial data \cite{Foxy1808}, and the H\"older distance estimate for
non-collapsing initial data in \cite{HKRX18}, we obtain the following H\"older
distance estimate for collapsing initial data \cite[Theorem~A.1]{HW20b}:
\begin{Theorem}\label{thm: HW20b2}
Given a positive integer $m$, positive constants $\bar{C}_0$, $C_R$, $T\le 1$
and $\alpha\in (0,1)$, there are constants
$C_D\big(\bar{C}_0,C_R,m\big)\ge 1$ and $T_D\big(\bar{C}_0,C_R,m\big)\in (0,T]$ such that for an
$m$-di\-men\-sio\-nal complete Ricci flow $(M,g(t))$ defined for $t\in [0,T]$, if
for some $x_0\in M$ and any $t\in [0,T]$ we have
\begin{gather*}
\Sc_{g(0)}\ge -C_R\quad \text{in}\ B_{g(0)}(x_0,10),\\ 
 |\Rc_{g(t)} |_{g(t)}\le \alpha t^{-1}\quad \text{in}\
B_{g(t)}\big(x_0,10+\sqrt{t}\big),
\end{gather*}
and the initial metric has a uniform bound $\bar{C}_0$ on the doubling and
Poincar\'e constant for the geodesic ball $B_{g(0)}(x_0,10)$, then
for any $x,y\in B_{g(0)}\big(x_0,\sqrt{T_D}\big)$ and $t\in [0,T_D]$, we
have
\begin{gather*}
C_D^{-1}d_{g(0)}(x,y)^{1+2\alpha}d_{g(0)}(x,y) \le d_{g(t)}(x,y) \le
C_Dd_{g(0)}(x,y)^{1-2\alpha}.
\end{gather*}
\end{Theorem}
Notice that the curvature assumption is natural (in view of the pseudo-locality
theorem) and the comparison with the initial time slice is the key difficulty
-- for positive time slices the Ricci curvature bound makes the estimate
trivial. Another handy distance distortion estimate for the application of
smoothing the collapsing initial is the following
\begin{Lemma}\label{lem: dis_dis1}
For any $\alpha\in (0,1)$, there is a positive quantity $\Psi_D(\alpha|m)$
with $\lim\limits_{\alpha\to 0}\Psi_D(\alpha|m)=0$, such that under the assumption of
Theorem~{\rm \ref{thm: HW20b1}} or Theorem~{\rm \ref{thm: HKRX18b}}, for any $x,y\in
B_g(p,2)$ and any $t\in \big(0,\varepsilon^2_P(m,\alpha)\big]$, if $d_g(x,y)\le
\sqrt{t}$, then we have
\begin{gather*}
|d_{g(t)}(x,y)- d_g(x,y)| \le \Psi_D(\alpha|m)\sqrt{t}.
\end{gather*}
\end{Lemma}

This lemma is a slight re-wording of \cite[Lemma~1.11]{HKRX18}, which concerns
non-collapsing initial data. See also \cite[Lemma~4.1]{HW20b} for a~proof.

\subsection{Applications of the Ricci flow local smoothing technique}\label{section4.4}
With the Ricci flow smoothing tool kit at hand (the flow existence time lower
bounds and the distance distortion estimates), we could in many cases reduce our
consideration of collapsing manifolds with Ricci curvature bounded below to the
classical collapsing geometry with bounded sectional curvature.

Locally, one could obtain infranil fiber bundle structure around points where
the Ricci flow smoothing results (Theorems~\ref{thm: HW20b1} and~\ref{thm: HKRX18b}) apply:
\begin{Theorem}\label{thm: local_infranil}
There is a positive constant $\delta=\delta(m)$ such that if $(M,g)$ is an
$m$-dimensional complete Riemannian manifold with $\Rc_g\ge -(m-1)g$, then for
any $p\in M$ which has a geodesic ball satisfying
\begin{align*}
d_{\rm GH}\big(B_g(p,2),\mathbb{B}^k(2)\big) < \delta
\end{align*}
and one of the following
conditions:
\begin{enumerate}\itemsep=0pt
 \item[$1)$] $p$ is a $(\delta,2)$-Reifenberg point~{\rm \cite{HKRX18}}, or
 \item[$2)$] $\rank \Gamma_{\delta}(p)=m-k$~{\rm \cite{HW20b}},
\end{enumerate}
there is an open neighborhood $U$ of $p$ such that $B_g(p,1-2\delta )\Subset
U\Subset B_g(p,1+2\delta)$, and $U$ is diffeomorphic to an infranil fiber
bundle over~$\mathbb{B}^k(1)$, with the extrinsic diameter of the fibers
bounded above by~$2\delta$.
\end{Theorem}
\begin{Remark}\label{rmk: orbifold_fibration}
In fact, as shown in \cite[Theorem~1.4]{HW20b}, there is
a positive constant $\delta=\delta(m,l)<1$ such that if $p\in M$ has a geodesic
ball satisfying $d_{\rm GH}\big(B_g(p,2),\mathbb{B}^k(2)\slash G\big) < \delta$
for some $G<{\rm O}(k)$ with $|G|\le l$, $\rank \Gamma_{\delta}(p)=m-k$, and there
exists a \emph{surjective} $\phi \in \operatorname{Hom}(\Gamma_{\delta}(p), G)$, then the same
infranil fiber bundle structure over the orbifold neighborhood
$\mathbb{B}^k(1) \slash G$ can be obtained. One can of course replace the
assumption on the nilpotency rank with the Reifenberg property as in item~(1).
See also \cite[Section~7]{Fukaya89} for related concepts.
\end{Remark}

Theorem~\ref{thm: local_infranil} with condition~(2) generalizes a local fiber
bundle result due to Naber and Zhang \cite[Proposition~6.6]{NaberZhang} from
the case of manifolds with bounded Ricci curvature to manifolds with Ricci
curvature only bounded from below. It is also a localization of \cite[Theorem~B]{HKRX18}. To prove this theorem, we first notice that for any $\alpha \in
(0,1)$ sufficiently small the assumptions enable us to run a Ricci flow with
the local initial data for a definite period of time, and obtain a smoothing
metric $g\big(\varepsilon^2\big)$ which is regular; by the distance distortion estimate
Lemma~\ref{lem: dis_dis1}, we know that up to scale $\varepsilon$, the original
metric structure defined by $g$ is $\Psi(\alpha)$-Gromov--Hausdorff close to the
metric structure defined by $g(\varepsilon)$; therefore, since the domain
$(B_g(p,2),g)$ is $\delta$-Gromov--Hausdorff close to~$\mathbb{B}^k(1)$, we know
that the domain $\big(B_g\big(p,\frac{3}{2}\big),g\big(\varepsilon^2\big)\big)$ is
$\delta+\Psi(\alpha)\varepsilon$-Gromov--Hausdorff close to~$\mathbb{B}^k(1)$ on
scales up to $\varepsilon$; but then the regularity of the metric
$g\big(\varepsilon^2\big)$ allows us to appeal to the classical theory of collapsing
geometry \cite[Theorem~2.6]{CFG92} with bounded sectional curvature to obtain
the infranil fiber bundle structure over~$\mathbb{B}^k(1)$.

Here we would like to emphasize that the classical theorems (e.g.,
Theorems~\ref{thm: Fukaya1} and~\ref{thm: CFG}, as well as \cite[Theorem~2.6]{CFG92}) on collapsing with bounded sectional curvature essentially
describe a gap phenomenon, rooted back in Gromov's almost flat manifold theorem
(Theorem~\ref{thm: Gromov78}): when the manifold is \emph{sufficiently}
Gromov--Hausdorff close to a lower dimensional space, then the manifold itself
already acquires some non-trivial symmetry. Such a gap phenomenon allows us to
slightly perturb the given metric locally to one with much better regularity,
but remains to be sufficiently collapsed (in the metric sense) so that the
symmetry structure could still be observed.

The Ricci flow local smoothing results can also help with proving global results
when the collapsing limit is singular. In particular, we make the following
\begin{Conjecture}
Given $D\ge 1$, $m,l\in \mathbb{N}$ and $\iota>0$ there is an
$\varepsilon(m,l,\iota)>0$ such that if $(M,g)\in \mathcal{M}_{\rm Rc}(m,D)$ and an
$(l,\iota)$-controlled $k$-dimensional Riemannian orbifold $(X,d_X)$ satisfy the
conditions $d_{\rm GH}(M,X)<\varepsilon$ and $b_1(M)-b_1(X)=m-k$, then $M$ is a
torus bundle over $X$.
\end{Conjecture}
Here by saying the Riemannian orbifold $(X,d_X)$ is
$(l,\iota)$-\emph{controlled} we mean that for any $x\in X$,
$B_{d_X}(x,\iota)\equiv \mathbb{B}^k\slash G_x$ with the order
of the orbifold group $G_x<{\rm O}(k)$ bounded above by $l$.

\section{Collapsing Ricci-flat K\"ahler metrics with bounded curvature}\label{section5}
Combining the classical theory of collapsing geometry with bounded curvature
and the Ricci flow smoothing technique, we make some attempts to understand open Question~\ref{qst: question}.

If $\big(M,\bar{g},\bar{J}\big)$ is a closed Calabi--Yau manifold with $\bar{g}$ being a
Ricci-flat K\"ahler metric, then by the work of Dai, Wang and Wei \cite{DWW07},
there is a $C^k$-\emph{stability radius} (in the space of Riemannian metrics on
$M$) $\bar{\eta}_1>0$ determined by $\bar{g}$, such that the Ricci
flow initiated from any Riemannian metric $g\in
B_{C^k(M,\bar{g})}(\bar{g},\bar{\eta}_1)$ converges to another Ricci-flat
K\"ahler metric in $B_{C^k(M,\bar{g})}(\bar{g},\bar{\eta}_0)$, with
$\bar{\eta}_1<\bar{\eta}_0<1$. On the other hand, if $(M,\bar{g})$ is
sufficiently volume collapsing with bounded sectional curvature, i.e., the
volume $|M|_{\bar{g}}<\varepsilon(m)$ as in Theorem~\ref{thm: CFG}, then we are
able to perturb~$\bar{g}$ to some nearby metric which is invariant under the
extra symmetry that causes collapsing. If now $|M|_{\bar{g}}<\varepsilon$ with
$\varepsilon$ so small that the constant $\Psi_{\rm CFG}(\varepsilon|m)$ is small
enough to guarantee the convergence of the Ricci flow starting from the
approximating metric, then the flow will enable us to find a Ricci-flat
K\"ahler metric compatible with the symmetry. This is the content of the
following theorem.
\begin{Theorem}\label{thm: main}
Given a closed K\"ahler manifold $\big(M,\bar{g},\bar{J}\big)$ equipped with a
Calabi--Yau metric $\bar{g}$ such that
$\max_{\wedge^2TM} |\mathbf{K}_{\bar{g}} |\le 1$, there is a constant
$\beta(\bar{g})>0$ such that if $\beta(\bar{g})<1$, then
\begin{enumerate}\itemsep=0pt
 \item[$1)$] there is a Ricci-flat K\"ahler metric $g$ $($together with a compatible
 complex structure $J)$, such that $ \|g-\bar{g}
 \|_{C^k(M,\bar{g})} <\bar{\eta}_0$ for some $\bar{\eta}_0\in (0,1)$
 solely determined by $\bar{g}$;
 \item[$2)$] there are a Ricci-flat orbifold $X$ and a Riemannian submersion $f\colon M\to
 X$ with respect to $g$, such that the fibers are totally geodesic tori $($see
{\rm \cite[Section~7]{Fukaya89}} for related definitions$)$, and~$g$ is invariant under the
 trous action.
\end{enumerate}
\end{Theorem}

Here we have $\beta(\bar{g}):=|M|_{\bar{g}}\bar{\eta}_3^{-1}$, with
$\bar{\eta}_3$ to be determined as following: notice that Cheeger, Fukaya and
Gromov's approximating metric is only in a $C^1$ neighborhood of $\bar{g}$, but
Dai, Wang and Wei's stability result requires much higher regularity for the
neighborhood; therefore we develop a regularity improvement tool
(Theorem~\ref{thm: main2}), which finds a $C^0$ neighborhood of $\bar{g}$,
denoted by $B_{C^0(M,\bar{g})}(\bar{g}, \bar{\eta}_2)$, where the Ricci flow
exists forever and converges to a Ricci-flat K\"ahler metric; here
$\bar{\eta}_3$ is defined so that if $|M|_{\bar{g}}<\bar{\eta}_3$, then
$\Psi_{\rm CFG}(\bar{\eta}_3|m)\le \frac{1}{2}\bar{\eta}_2$.

While the assumption on the volume collapsing is rather strong, in that
$\beta(\bar{g})$ depends on the specific K\"ahler manifold, the
existence of an invariant critical metric drastically reduces the topological
complexity of the manifold: it is a torus bundle over a Ricci-flat orbifold. The
invariant metric allows us to apply the O'Neill's formula~\cite{ONeill},
together with the central density (see \cite[Section~5]{Foxy1908}) of the collapsing
structure to rule out the so-called \emph{corner singularities} of the
collapsing limit space, and following the arguments in~\cite{Lott10, NT18}
we can show that the collapsing fibers must be tori and the fibration must
locally be a Riemannian product, implying the Ricci-flatness of the collapsing
limit.

\subsection{Existence of invariant Ricci-flat K\"ahler metric}\label{section5.1}
In this sub-section we prove the first claim in Theorem~\ref{thm: main}.
By the stability result \cite{DWW07} of Dai, Wang and Wei for Ricci-flat
K\"ahler metrics, we know that for the Calabi--Yau manifold $\big(M,\bar{g},
\bar{J}\big)$, there are some positive constants $\bar{\eta}_1
<\bar{\eta}_0<1$, both determined by $\bar{g}$, such that if $g$ is
another smooth Riemannian metric with $ \|g-
\bar{g} \|_{C^k(M,\bar{g})} <\bar{\eta}_1$, then the
Ricci flow with initial data $g$ exists for all time and converges to a
Ricci-flat K\"ahler metric in $B_{C^k(M,\bar{g})}(\bar{g},
\bar{\eta}_0)$. Here $k:=\big\lceil \bar{\eta}_1^{-1}
\big\rceil$ is solely determined by $\bar{g}$. Since $k$ may be a very large
number, our first priority is to prove the following regularity improvement
result.
\begin{Theorem}\label{thm: main2}
There is a constant $\bar{\eta}_2\in (0,\bar{\eta}_1)$
determined by $\bar{g}$, such that if $\|g-\bar{g}
\|_{C^0(M,\bar{g})} <\bar{\eta}_2$, then the Ricci
flow starting from $g$ exists for all time and converges to a Ricci-flat
K\"ahler metric in $B_{C^{k}(M,\bar{g})}(\bar{g},\bar{\eta}_0)$.
\end{Theorem}
Before proving Theorem~\ref{thm: main2}, we state the following slight variant
of Lemma~\ref{lem: dis_dis1}:
 \begin{Lemma} \label{lem: dis_dis}
 There is a positive constant $\bar{\eta}_2'<1$ determined by $\bar{g}$ such
 that if $g\in B_{C^0(M,\bar{g})}(\bar{g},\bar{\eta}_2')$, and the Ricci flow
 $g(t)$ with initial data $g$ satisfies $\max_M\left|\Rc_{g(t)} \right|_{g(t)}
 <\alpha t^{-1}$ for $t\in [0,T)$ and some $\alpha\in (0,1)$, then there is a
 $\Psi_D'(\alpha|\bar{g})>0$ with $\lim\limits_{\alpha\to 0}\Psi_D'(\alpha|\bar{g})=0$
 such that for any $x,y\in M$ with $d_g(x,y)\le \sqrt{t}$, we have
 \begin{align*}
 |d_g(x,y)-d_{g(t)}(x,y) | \le \Psi_D'(\alpha|\bar{g})\sqrt{t}.
 \end{align*}
\end{Lemma}
The proof of this lemma is the same as that of \cite[Lemma~2.10]{CRX19}: in
the contradiction argument involved, the rescaling limit will be exactly
$\mathbb{R}^n$ as the quantity $\Psi_D'(\alpha|\bar{g})$ depends on~$\bar{g}$.
With such distance distortion estimate at our disposal, we now prove the $C^0$
stability around a~stable Ricci-flat K\"ahler metric.
 \begin{proof}[Proof of Theorem~\ref{thm: main2}]
 We notice that since $M$ is a closed manifold, for
 any smooth Riemannian metric $g$ there is a Ricci flow solution with initial
 data~$g$. Moreover, for any $\alpha>0$, there is alway some $\varepsilon>0$
 (depending on $\alpha$ and $(M,g)$) such that the following curvature bound is
 valid for $t\in \big(0,\varepsilon^2\big]$:
\begin{align*}
 \|\Rm_{g(t)} \|_{C^0(M,g(t))} \le \alpha t^{-1}.
\end{align*}
We now make the following claim regarding the Ricci flow solution:
\begin{Claim}
There are some $\alpha \in (0,1)$ and $\delta\in (0,\bar{\eta}_2')$ depending on
$(M,\bar{g})$ such that if
\begin{align*} \|g-\bar{g} \|_{C^0(M,\bar{g})} < \delta\qquad
\text{and}\qquad \forall\, t\le
\varepsilon^2,\quad \|\Rm_{g(t)} \|_{C^0(M,g(t))} \le \alpha t^{-1},
\end{align*}
then $\big\|g\big(\varepsilon^2\big)-\bar{g}\big\|_{C^k(M,\bar{g})}<\bar{\eta}_1$.
\end{Claim}
\begin{proof}[Proof of the claim]
We now prove this claim via a contradiction argument: if the theorem fails,
 we may find a sequence of smooth Riemannian metrics~$g_i$ on~$M$ and sequences
 of positive numbers $\alpha_i\to 0$ and $\delta_i\to 0$, such that
 $ \|g_i-\bar{g} \|_{C^0(M,\bar{g})} =\delta_i\to 0$ as $i\to \infty$,
 the Ricci flow solutions $g_i(t)$ satisfy
\begin{align}\label{eqn: Rm_bound_claim}
\forall\, t\le \varepsilon_i^2,\quad \|\Rm_{g_i(t)} \|_{C^0(M,g_i(t))}
\le \alpha_i t^{-1}.
\end{align}
for some $\varepsilon_i\in (0,1)$, but $\big\|g_i\big(\varepsilon_i^2\big)
-\bar{g}\big\|_{C^k(M,\bar{g})} \ge \bar{\eta}_1$ for all $i$ sufficiently
large.

From the contradiction hypothesis we may find points $p_i\in M$ such that for
all $i$ large enough,
\begin{align}\label{eqn: C3_contradiction}
\big|g_i\big(\varepsilon_i^2\big)-\bar{g}\big|_{C^k}(p_i) \ge \bar{\eta}_1.
\end{align}

On the other hand, by the curvature control (\ref{eqn: Rm_bound_claim}) and
Shi's estimates we have
\begin{align}\label{eqn: regularity}
\forall\, l\in \mathbb{N},\quad
\big\|\nabla^l\Rm_{g_i(\varepsilon_i^2)}\big\|_{C^0(M,g_i(\varepsilon_i^2))}
\le C_l \varepsilon_i^{-2-2l},
\end{align}
with $C_l>0$ depending only on $m$ for $l>1$, and $C_0=\alpha_i$.
Moreover, applying Lemma~\ref{lem: dis_dis} with~(\ref{eqn: Rm_bound_claim}) we
have the distance distortion estimate for any $x,y\in M$ with $d_{g_i}(x,y)\le
\varepsilon_i$:
\begin{align*}
\big|d_{g_i}(x,y)-d_{g_i(\varepsilon_i^2)}(x,y)\big| \le
\Psi_D'(\alpha_i|\bar{g})\varepsilon_i.
\end{align*}

We now rescale the metrics so that $\varepsilon_i\mapsto 2$, setting
$h_i(t):=4\varepsilon_i^{-2}g_i\big(\varepsilon_i^{-2} t\big)$ for $t\in
\big[0,\varepsilon_i^2\big]$ and $\bar{h}_i:=4\varepsilon_i^{-2}\bar{g}$, and
estimate for all $x,y\in M$ with $d_{h_i(0)}(x,y)\le 2$, that
\begin{align*}
 |d_{\bar{h}_i}(x,y)-d_{h_i(1)}(x,y) | & \le
 |d_{\bar{h}_i}(x,y)-d_{h_i(0)}(x,y) |
+ |d_{h_i(0)}(x,y)-d_{h_i(1)}(x,y) |\\
& \le \Psi(\delta_i)+\Psi_D'(\alpha_i|\bar{g}),
\end{align*}
which approaches $0$ as $i\to \infty$.

Consequently, the we have the pointed Gromov--Hausdorff distance estimate
\begin{align}\label{eqn: GH}
d_{p{\rm GH}}\left(B_{h_i(1)}(p_i, 1), B_{\bar{h}_i}(p_i, 1)
\right) < \Psi(\delta_i)+\Psi_D'(\alpha_i|\bar{g}),
\end{align}
and such Gromov--Hausdorff distance bounds are realized by the identity map.

Since $\{\varepsilon_i\}\subset [0,1]$ and $\{p_i\}\subset M$, by the
compactness we may pass to convergent sub-sequences, still denoted by
$\{\varepsilon_i\}$ and $\{p_i\}$, respectively. We denote
$\lim_i\varepsilon_i=\varepsilon$ and \mbox{$\lim_ip_i=p\!\in\! M$}. There are two
possibilities: either $\varepsilon_i\to\varepsilon>0$ or $\varepsilon=0$.

If $\varepsilon>0$, we notice that
$\bar{h}_i=4\varepsilon_i^{-2}\bar{g}\to 4\varepsilon^{-2}\bar{g}$ smoothly as
$i\to\infty$. Moreover, since $M$ is compact, the collection of geodesic balls
$ \{B_{\bar{h}_i}(p_i,1) \}$ has uniformly bounded geometry -- the
geometry is ultimately bounded by that of
$\overline{B_{4\varepsilon^{-2}\bar{g}}(p,2)}\subset M$. Consequently, there is
a smooth limit metric $\bar{h}_{\infty}$ such that as $i\to \infty$,
\begin{align*}
\big(B_{\bar{h}_i}(p_i,1),\bar{h}_i\big) \xrightarrow{p{\rm CG}}
\big(B_{4\varepsilon^{-2}\bar{g}}(p,1),\bar{h}_{\infty}\big).
\end{align*}
Here $p{\rm CG}$ means pointed Cheeger--Gromov (smooth) convergence. On the other hand,
by~(\ref{eqn: GH}), it is clear that as $i\to \infty$,
\begin{align*}
\big(B_{h_i(1)}(p_i,1),h_i(1)\big) \xrightarrow{p{\rm GH}}
\big(B_{4\varepsilon^{-2}\bar{g}}(p,1), \bar{h}_{\infty}\big).
\end{align*}
However, since the metrics $\{h_i(1)\}$ has uniformly controlled regularity~(\ref{eqn: regularity}), this last convergence must also be in the pointed
Cheeger--Gromov sense. Say, the limit metric $h_{\infty}$ is defined on
$B_{4\varepsilon^{-2}\bar{g}}(p,1)$ and we have $h_{\infty}\equiv
\bar{h}_{\infty}$. But the pointed Cheeger--Gromov convergence also implies, by~(\ref{eqn: C3_contradiction}), that{\samepage
\begin{align*}
 \big|h_{\infty}-\bar{h}_{\infty} \big|_{C^k}(p) = \lim_{i\to\infty}
 \big|h_i(1)-\bar{h}_i \big|_{C^k}(p_i) \ge \bar{\eta}_1,
\end{align*}
and this contradicts the conclusion $h_{\infty}\equiv \bar{h}_{\infty}$.}

If $\varepsilon=0$, since $\bar{h}_i$ is nothing but the metric $\bar{g}$ scaled
up, $M$ is compact and $\varepsilon_i\to 0$, it is clear that
$B_{\bar{h}_i}(p_i,1) \xrightarrow{p{\rm CG}} \mathbb{B}^m(1)$ as $i\to \infty$. By~(\ref{eqn: GH}), we then have $B_{h_i(1)}(p_i,1) \xrightarrow{p{\rm GH}}
\mathbb{B}^m(1)$ as $i\to \infty$. However, the uniform regularity control~(\ref{eqn: regularity}) improves the convergence also to pointed smooth
Cheeger--Gromov convergence, i.e., as $i\to \infty$ we in fact have
\begin{align*}
\big(B_{h_i(1)}(p_i,1),h_i(1)\big) \xrightarrow{p{\rm CG}} \big(\mathbb{B}^m(1),h_{\infty}\big),
\end{align*}
with the limit smooth metric $h_{\infty}\equiv g_{\rm Euc}$.
But by (\ref{eqn: C3_contradiction}) and (\ref{eqn: regularity}) we have
\begin{align*}
 \big|h_{\infty}-g_{\rm Euc} \big|_{C^k}(o) = \lim_{i\to\infty}
 \big|h_i(1)-\bar{h}_i \big|_{C^k}(p_i) \ge \bar{\eta}_1,
\end{align*}
which is a contradiction.
\end{proof}
We now return to the proof of the theorem. Fix the $\alpha$ and
$\delta=:\bar{\eta}_2$ determined by $\bar{g}$ through the claim,
then for any smooth Riemannian metric $g$ with $ \|g-\bar{g}
 \|_{C^0(M,\bar{g})}<\bar{\eta}_2$, we see that the evolved metric
$g\big(\varepsilon^2\big)\in B_{C^k(M,\bar{g})}(\bar{g},\bar{\eta}_1)$, and thus the
Ricci flow continuing from $g\big(\varepsilon^2\big)$ exists for all time and converges
to a Ricci-flat K\"ahler metric in $B_{C^k(M,\bar{g})}(\bar{g},\bar{\eta}_0)$.
\end{proof}

The proof of the first claim in Theorem~\ref{thm: main} is now a simple
combination of Theorems~\ref{thm: main2} and~\ref{thm: CFG}.
 \begin{proof}[Proof of Theorem~\ref{thm: main}(1)]
By Theorem~\ref{thm: CFG} we may find some $\bar{\eta}_3\in (0,\varepsilon(m))$
determined by $\bar{g}$ with $\Psi_{\rm CFG}(\bar{\eta}_3|m)\le
\frac{1}{2}\bar{\eta}_2$, and set $\beta(\bar{g}):= \bar{\eta}_3^{-1}|
M|_{\bar{g}}$. Since $\max_M\left|\mathbf{K}_{\bar{g}} \right|_{\bar{g}}\le 1$,
if we have $\beta(\bar{g})<1$, then the original work of Cheeger, Fukaya and
Gromov enables us to find an approximating metric $g' \in
B_{C^0(M,\bar{g})}\left(\bar{g}, \bar{\eta}_2\right)$ which is $(\rho,k)$-round
and compatible with a nilpotent Killing structure $\mathfrak{N}$ whose orbits
are of diameter less than $\bar{\eta}_2$. Now by Theorem~\ref{thm: main2} and
the fact that $g'\in B_{C^0(M,\bar{g})}(\bar{g}, \bar{\eta}_2)$, the Ricci flow
starting from $g'$ exists for all time and converges to a Ricci-flat K\"ahler
metic $g:=g'(\infty)\in B_{C^k(M,\bar{g})} (\bar{g},\bar{\eta}_0)$. Since the
Ricci flow is intrinsic, the infinitesimal isometries are preserved (i.e., it
preserves the Killing vector fields). Consequently, $\mathfrak{N}$ is still a
nilpotent Killing structure compatible with $g$, which is our desired metric.
\end{proof}

\subsection{Reduction of the diffeomorphism type}\label{section5.2}
In this sub-section we prove the second claim of Theorem~\ref{thm: main}.
Under the assumption that $\beta(\bar{g})<1$ and by the choice of
$\bar{\eta}_3$, we see that $M$ admits a nilpotent Killing structure
$\mathfrak{N}$. Notice that $\mathfrak{N}$ is compatible with $g$ in that
the local nilpotent group actions on $M$ are isometric with respect to $g$. In
fact, since $M$ is very collapsed with bounded diameter, the structure is
\emph{pure}, i.e., there is a single nilpotent Lie algebra $\mathfrak{n}$ such
that the germ of the acting Lie group at every point has its identity component
generated by $\mathfrak{n}$. Moreover, this gives us a singular Riemannian
submersion $f\colon M\to X$ over some collapsing limit space~$(X,d_X)$ whose
topologcial structure is described by Theorem~\ref{thm: singular_fibration}.

In fact, the nilpotent Killing structure $\mathfrak{N}$ determines a (singular)
Riemannian foliation $\mathcal{N}$ (see~\cite{Molino88}), defined by the
distribution of Killing vector fields tangent to the orbits of the structure~$\mathfrak{N}$. Clearly, the leaf of~$\mathcal{N}$ passing through~$p$ is~$\mathcal{O}(p)$, the orbit of~$p$ under the nilpotent Killing structure
$\mathfrak{n}$, and it is also a component of the fiber~$f^{-1}(f(p))$. Within
$\mathcal{N}$ we can define the central distribution~$\mathcal{F}$, consisting
of the center $C(\mathcal{N}_p)\trianglelefteq \mathcal{N}_p$ at every $p\in
M$. This defines another (singular) Riemannian foliation by the Frobinius
theorem -- in fact, this defines an $F$-structure $\mathfrak{F}$ \emph{a la}
Cheeger and Gromov~\cite{CGI, CGII}.

The leaf space of $\mathcal{N}$ (or equivalently the orbit space of
$\mathfrak{N}$) is isometric to the collapsing limit $(X,d_X)$, and the
leaf space of $\mathcal{F}$ is isometric to a metric space $(W,d_W)$. Recalling
our descriptions in Section~\ref{section2.1} (according to \cite[Theorem~0.5]{Fukaya88}), we have
$X=\tilde{\mathcal{R}}\sqcup \tilde{\mathcal{S}}$, where $\tilde{\mathcal{R}}$
is an open (incomplete) Riemannian orbifold, and $\tilde{\mathcal{S}}$ consists
of \emph{corner points}, i.e., those $x\in X$ with $\dim G_x>0$. For such an~$x$,
 $G_x$ has its non-trivial identity component as a torus, and more
 significantly, for any $p\in f^{-1}(x)$, the infinitesimal action of $G_x$ is
 contained in $\mathcal{F}_p=C(\mathcal{N}_p)$; see \cite[Lemma~5.1]{Fukaya88}.
 Let $\mathfrak{g}_p$ denote the Lie algebra of $G_x$ for any $p\in f^{-1}(x)$,
 then $\mathfrak{g}_p\cong \mathbb{R}^d$ as Lie algebras, with $d=0$ when $x\in
 \tilde{\mathcal{R}}$ and $d\in \left\{1,\ldots,m-\dim X\right\}$ when $x\in
 \tilde{\mathcal{S}}$.

Clearly, $\mathcal{N}_p$, the distribution $\mathcal{N}$ located at each $p\in
M$, is isomorphic to $\mathfrak{n}\slash \mathfrak{g}_p$ as Lie algebras, and
$\mathcal{F}_p$ is isomorphic to $C(\mathfrak{n})\slash \mathfrak{g}_p$ since
$\mathfrak{g}_p\trianglelefteq C(\mathfrak{n})$. Consequently, the foliation
$\mathcal{N}$ is a Riemannian foliation (i.e., non-singular) \emph{if and only
if} $\tilde{\mathcal{S}}=\varnothing$ (which is also equivalent to saying that the
nilpotent Killing structure $\mathfrak{N}$ is \emph{polarized}), and the same
conclusion holds for $\mathcal{F}$ and $\mathfrak{F}$. In fact, by
\cite[Theorem~0.5]{Fukaya88}, the possibly singular Riemannian foliations
$\mathcal{N}$ and $\mathcal{F}$ are always linearized; see~\cite{Molino88}. We
will denote $k'=\dim C(\mathfrak{n})\le m-\dim X$.

The $\mathfrak{N}$ (and thus $\mathfrak{F}$) invariant metric $g$, when
restricted to a leaf of $\mathcal{F}$, defines a non-negative definite
$2$-tensor field~$G$: for any two vector fields $\xi,\zeta \in
\Gamma(\mathcal{F},M)$, $G(\xi,\zeta):=g(\xi,\zeta)$, which is left
invariant along the leaves of $\mathcal{F}$. Consequently, the \emph{central
density}~-- $\det G$~-- is a~non-negative basic function for $\mathcal{F}$,
i.e., it descends to a non-negative function on $W$; see also~\cite{Foxy1908}.
Clearly, $\det G$ is smooth around the regular leaves of $\mathcal{F}$, i.e.,
those leaves whose tangents are isomorphic to $C(\mathfrak{n})$. By the
previous discussion on $\mathcal{F}_p$ for $p\in f^{-1}\big(\tilde{\mathcal{S}}\big)$,
we see that $\det G$ vanishes exactly on $f^{-1}\big(\tilde{\mathcal{S}}\big)$, i.e.,
$ \{p\in M\colon \det G(p)=0 \}=f^{-1}\big(\tilde{\mathcal{S}}\big)$. This is
because at $p\in f^{-1}\big(\tilde{\mathcal{S}}\big)$ we can extend $G$ by $0$ on
$\mathcal{F}_p\oplus \mathfrak{g}_p\cong C(\mathfrak{n})$; compare also
\cite[Theorem~0.6]{Fukaya87mGH}. Notice that $G$ induces a left invariant
Riemannian metric on the leaves of $\mathcal{F}$; but since $\mathcal{F}$
consists of commuting vector fields, $G$ is bi-invariant and actually flat
along the leaves of~$\mathcal{F}$.

In the formulas below, we will use the Roman letters $i$, $j$, $k$, $l$ to index the
coordinates along the leaf directions of $\mathcal{F}$, and for directions
perpendicular to a leaf of $\mathcal{F}$, we use the Greek letters
$\alpha$, $\beta$ as indices. Moreover, we employ the Einstein
summation convention, adding the repeated indices. As the basic function
$\det G$ is constant along the leaves of $\mathcal{F}$, if it were not a~constant throughout $M$, then $\max_M \ln \det G=\max_W \ln \det G$ is attained
at some $p\in f^{-1}\big(\tilde{\mathcal{R}}\big)$, since $\det G\ge 0$ and $\det
G|_{f^{-1}(\tilde{\mathcal{S}})}\equiv 0$. By the O'Neill's formula~\cite{ONeill} applied to the (singular) Riemannian foliation $\mathcal{F}$ in a
small enough open neighborhood around $p\in M$ and the flatness of the leaves,
we have some locally defined basic $1$-form $A_{\alpha\beta}^l$ such that
\begin{align}\label{eqn: ONeill_G}
\Rc_{ij} = -\tfrac{1}{2} ( G_{ij;\alpha\alpha}
-G_{ik,\alpha}G_{jk,\alpha} ) -\tfrac{1}{4}G^{kl}G_{kl,\alpha}
G_{ij,\alpha}+\tfrac{1}{4}G_{ik}G_{jl}A^k_{\alpha\beta}A^l_{\alpha\beta}.
\end{align}
Now tracing by $G^{ij}$ on the leaf directions and by the Ricci-flatness of the
metric on $M$, we see that
\begin{align}\label{eqn: detG}
\tfrac{1}{2}\Delta^{\perp}\ln \det G+\tfrac{1}{4}\big|\nabla^{\perp} \ln \det G\big|^2 = |A|^2.
\end{align}
Here $\nabla^{\perp}$ and $\Delta^{\perp}$ denote the derivatives taken
perpendicular to the leaf directions. However, since $\ln \det G(p)=\max_M \ln
\det G$, we must have $\Delta^{\perp}\ln\det G(p)<0$ and $\nabla^{\perp}\ln
\det G(p)=0$ -- this will contradict the non-negativity of the right-hand side
of~(\ref{eqn: detG}). Consequently, we know that $\det G$ is a positive
constant. This implies that $\tilde{\mathcal{S}}=\varnothing$, i.e.,
$X=\tilde{\mathcal{R}}$ is a Riemannian orbifold. Moreover, $\mathcal{F}$ is a~Riemannian foliation -- the orbit space~$W$ of $\mathcal{F}$ is
also a Reimannian orbifold.

Clearly, $G_{ij}$ descends to germs of smooth functions on the regular part of
$W$; if $w\in W$ is a~singularity, it can only be an orbifold point and we may
pull the corresponding quantities back to its local orbifold covering -- the
differentials of $G_{ij}$ are well-defined throughout~$W$. Notice that
$\nabla_W G_{ij}= \nabla^{\perp}G_{ij}$ and $\Delta_W G_{ij}=
\Delta^{\perp}G_{ij}$. Moreover, by~(\ref{eqn: detG}) the constancy of $\ln
\det G$ ensures that $|A|\equiv 0$ throughout $X$, and thus by~(\ref{eqn:
ONeill_G}),
\begin{align*}
\Delta_W G_{ij} = \Delta^{\perp}G_{ij} = G_{ik,\alpha}G_{jk,\alpha}.
\end{align*}
If we further check the O'Neill's formula for the regular part of $W$, the
vanishing of $\big|\nabla^{\perp}\ln \det G\big|$ and~$|A|$, together with
the Ricci-flatness of $M$ tell that
\begin{align}\label{eqn: Rc_Y}
 (\Rc_W )_{\alpha\beta} = \tfrac{1}{4}G_{ij,\alpha}G_{ij,\beta}.
 \end{align}
We notice that the quantity $|\nabla_W G|^2$ is a globally defined non-negative
smooth function on $W$. By the compactness of~$W$, if
$\left|\nabla_W G\right|^2\not\equiv 0$, then we have $\Delta_W |\nabla_W
G |^2(w)<0$ at the maximum point $w\in W$. On the other hand, since $\det
G$ is a constant, we may view the $ [G_{ij} ]$ as a matrix valued map $G\colon
U \to {\rm SL}(k',\mathbb{R}) \slash {\rm SO}(k')$, with $U\subset W$ being an open
 neighborhood of $w$ where we can write down $G$ in coordinates. We now notice
 that the codomain $S_{k'}:={\rm SL}(k',\mathbb{R}) \slash {\rm SO}(k')$ is a negatively
 curved symmetric space. Now by the Bochner formula, we can calculate at $w\in
 U$ to see
 \begin{align*}
\Delta_W |\nabla_W G |^2(w)& =
2 |\nabla_W\nabla_W G |^2(w)+
 (\Rc_W )_{\alpha\beta}G_{ij,\alpha}G_{ij,\beta}(w)
- (G^{\ast}\Rm_{S_{k'}} )_{\alpha\beta\beta\alpha}(w)\\
 & >\tfrac{1}{4} |\nabla_W G |^4(w) > 0,
 \end{align*}
 which is impossible. Therefore, we see the tensor $\nabla^{\perp}G\equiv 0$
 on~$M$. Consequently, we see that~(\ref{eqn: Rc_Y}) reduces to
 $\Rc_W \equiv 0$. Moreover, as $\nabla^{\perp}G$ stands for the second
 fundamental form of the leaves of $\mathcal{F}$, its vanishing tells that the
 the leaves of $\mathcal{F}$ are totally geodesic. So the Riemannian
 metric $g$ locally splits, and by the left invariance of $g$ with respect to
 $\mathfrak{N}$, this implies the splitting of the nilpotent Killing structure,
 i.e., passing through each point $p\in M$, we have the splitting of Lie algebras
 $\mathcal{N}_p=\mathcal{F}_p\oplus \mathcal{N}_p'$ for some nilpotent Lie
 algebra $\mathcal{N}_p'$. This however leads to $\mathcal{F}=\mathcal{N}$:
 otherwise, $\mathcal{N}_p'\not= 0$ at some $p\in M$, and by the nilpotency it
 has to have a~\emph{non-trivial} center $C(\mathcal{N}_p')$, but the above
 splitting shows that $\mathcal{F}_p =C(\mathcal{N}_p)=\mathcal{F}_p\oplus
 C(\mathcal{N}_p')$, which is absurd. Consequently, we have
 $\mathcal{F}=\mathcal{N}$, $X\equiv W$ as Ricci-flat orbifolds, and that
 $M$ fibers over $X$ by flat tori with totally geodesic fibers (the leaves of~$\mathcal{N}$).

\subsection*{Acknowledgements}
The second author was partially supported by NSFC Grant 11821101, Beijing
Natural Science Foundation Z19003, and a research fund from Capital Normal
University. The third author is partially supported by the General Program of
the National Natural Science Foundation of China (Grant No.~11971452) and a
research fund of USTC. The authors would like to thank anonymous referees for
their careful proofreading and helpful comments of the paper.

\pdfbookmark[1]{References}{ref}
\LastPageEnding

\end{document}